\documentclass[twoside,11pt]{article}

\usepackage{blindtext}

\usepackage{jmlr2e}
\usepackage{algorithm}
\usepackage{algorithmicx}
\usepackage{graphicx}
\usepackage{algpseudocode}
\usepackage{threeparttable}
\usepackage{multirow}
\usepackage{amssymb,amsmath}
\usepackage{paralist}
\usepackage{subfigure}
\usepackage{color}

\newtheorem{assumption}{Assumption}
\newcommand{\BE}{\mathbb{E}}
\newcommand{\nx}{\nabla_{x}}
\newcommand{\ny}{\nabla_{y}}
\newcommand{\atwo}{\alpha^{2}}

\newcommand{\rl}{\|^{2}}
\newcommand{\fu}{f_{\mu}}
\newcommand{\xt}{x_{t}}
\newcommand{\yt}{y_{t}}
\newcommand{\xtt}{x_{t+1}}

\newcommand{\vt}{s_{t}}
\newcommand{\fa}{\frac{\alpha}{2}}

\newcommand{\np}{\nabla\Phi}

\usepackage{lastpage}
\jmlrheading{?}{2022}{1-\pageref{LastPage}}{12/21; Revised /}{/}{21-0000}{Zi Xu, Zi-Qi Wang, Jun-Lin Wang and Yu-Hong Dai}

\ShortHeadings{Zeroth-Order AGDA Algorithms for Nonconvex-Nonconcave Minimax Problems}{Xu, Wang, Wang and Dai}
\firstpageno{1}

\begin{document}

\title{Zeroth-Order Alternating Gradient Descent Ascent Algorithms for a Class of Nonconvex-Nonconcave Minimax Problems\thanks{This work is supported by National Natural Science Foundation of China under the grants
		12071279, 12021001, 11991021 and by General Project of Shanghai Natural Science Foundation
		(No. 20ZR1420600). }}

\author{\name Zi Xu \email xuzi@shu.edu.cn \\
       \addr Department of Mathematics\\
          Shanghai University\\
       Shanghai 200444, People's Republic of China\\
       and\\
       Newtouch Center for Mathematics of Shanghai University\\
       Shanghai University\\
       Shanghai 200444, People's Republic of China
       \AND
       \name Zi-Qi Wang \email colwzq@shu.edu.cn \\
       \addr Department of Mathematics\\
       Shanghai University\\
       Shanghai 200444, People's Republic of China
	   \AND
	   \name Jun-Lin Wang \email wjl37@shu.edu.cn \\
	   \addr Department of Mathematics\\
	   Shanghai University\\
	   Shanghai 200444, People's Republic of China
		\AND
		\name Yu-Hong Dai\thanks{Corresponding author.}  \email dyh@lsec.cc.ac.cn \\
		\addr LSEC, ICMSEC, Academy of Mathematics and Systems Science,
		Chinese Academy of Sciences\\
		Academy of Mathematics and Systems Science,
		Chinese Academy of Sciences\\
		Beijing 100190, China}

\editor{My editor}

\maketitle

\begin{abstract}
In this paper, we consider a class of nonconvex-nonconcave minimax problems, i.e., NC-PL minimax problems, whose objective functions satisfy the Polyak-\L ojasiewicz (PL) condition with respect to the inner variable. We propose a zeroth-order alternating gradient descent ascent (ZO-AGDA) algorithm and a zeroth-order variance reduced alternating gradient descent ascent (ZO-VRAGDA) algorithm  for solving NC-PL minimax problem under the deterministic and the stochastic setting, respectively. The total number of function value queries to obtain an $\epsilon$-stationary point of ZO-AGDA and ZO-VRAGDA algorithm for solving NC-PL minimax problem is upper bounded by $\mathcal{O}(\varepsilon^{-2})$ and $\mathcal{O}(\varepsilon^{-3})$, respectively. To the best of our knowledge, they are the first two zeroth-order algorithms with the iteration complexity gurantee for solving NC-PL minimax problems. 
\end{abstract}


\section{Introduction}

Consider nonconvex minimax problems under both the deterministic setting, i.e.,
\begin{align}\label{min_max_problem_deterministic}
	\min_{x} \max_{y} f(x,y),
\end{align}
and the stochastic setting with the objective function being an expectation function, i.e.,
\begin{align}
	\min_{x} \max_{y} g(x,y) = \mathbb{E}_{\xi \sim \mathcal{P}}G(x,y;\xi),\label{min_max_problem_stochastic}
\end{align}
where $f(x, y), G(x,y;\xi): \mathbb{R}^{d_1}\times\mathbb{R}^{d_2}\rightarrow \mathbb{R}$ are smooth functions, possibly nonconvex in variable $x$ and nonconcave in variable $y$, $\xi$ is a random variable following an unknown distribution $\mathcal{P}$, and $\mathbb{E}$ denotes the expectation function. 

Recently, many applications such as adversarial attacks on deep neural networks (DNNs), reinforcement learning, robust training, hyperparameters tuning and bandit convex optimization in machine learning or deep learning fields~\citep{chen2017zoo,finlay2021scaleable,snoek2012practical} are nonconvex minimax optimization problems in which only the objective function but not the gradient information is available. In this paper, we consider the  same setting that the gradient of the function in the problems \eqref{min_max_problem_deterministic} and \eqref{min_max_problem_stochastic} cannot be obtained directly, and the corresponding algorithm is called zeroth-order algorithm.

There are some existing works that focus on zeroth-order algorithms for solving minimax optimization problems under the nonconvex-strongly concave setting. For example, for solving \eqref{min_max_problem_deterministic} (resp. \eqref{min_max_problem_stochastic}), Wang et al.~\citep{wang2020zerothorder} proposed two single-loop algorithms, i.e., ZO-GDA and ZO-GDMSA (resp. ZO-SGDA and ZO-SGDMSA), and the total number of calls of the zeroth-order oracle to obtain an $\epsilon$-stationary point is bounded by $\mathcal{O}(\kappa^{5}(d_{1}+d_{2}) \epsilon^{-2})$ (resp. $\mathcal{O}(\kappa^{5}(d_{1}+d_{2}) \epsilon^{-4})$) and $\mathcal{O}(\kappa(d_{1}+\kappa d_{2} \log (\epsilon^{-1})) \epsilon^{-2})$ (resp. $\mathcal{O}(\kappa(d_{1}+\kappa d_{2}\log (\epsilon^{-1})) \epsilon^{-4})$) respectively, where $\kappa$ is the condition number. For solving \eqref{min_max_problem_stochastic}, \citet{liu2020min} proposed an alternating projected stochastic gradient descent-ascent method called ZO-Min-Max which can find an $\varepsilon$-stationary point with the total complexity of  $\mathcal{O}((d_{1}+d_{2}) \varepsilon^{-6})$. There are some multi-loop algorithms combined with variance-reduced or momentum techniques are also proposed to solve \eqref{min_max_problem_stochastic}. For example, \citet{xu2020enhanced} proposed a variance reduced gradient descent ascent (ZO-VRGDA) algorithm, which achieves the total complexity of $\mathcal{O}( \kappa^{3}( d_{1}+d_{2}) \varepsilon^{-3})$, and \citet{huang2020accelerated} proposed an accelerated momentum-based descent ascent (Acc-ZOMDA) method with the total complexity of $\mathcal{O}(\kappa^{3}(d_{1}+d_{2})^{3 / 2} \epsilon^{-3})$.

Under nonconvex-concave setting, the only zeroth-order algorithm that we know is the ZO-AGP algorithm proposed by \citet{xu2021zeroth} and its iteration complexity is bounded by $\mathcal{O}(\varepsilon^{-4})$ with the number of function value estimation per iteration being bounded by $\mathcal{O}(d_{1}+d_{2})$.

\subsection{Related Works}

We give a brief review on first-order algorithms for solving minimax optimization problems. For solving convex-concave minimax optimization problems, there are many existing works. For instance,  \citet{nemirovski2004prox} proposed a mirror-prox algorithm and \citet{nesterov2007dual} proposed a dual extrapolation algorithm to solve smooth convex-concave minimax problems. An extra-gradient algorithm and an optimistic gradient descent ascent algorithm to solve bilinear and strongly convex-strongly concave minimax optimization problems were proposed in~\citep{mokhtari2020unified}, and both of them own the iteration comlexity of $\mathcal{O}( \kappa \log(1/\varepsilon) )$.  \citet{lin2020near} proposed a near optimal algorithm for solving convex-concave minimax optimization problems, which achieves a iteration complexity of $\mathcal{O}( \varepsilon^{-1} )$. For nonconvex-strongly concave minimax problems,  \citet{luo2020stochastic}  proposed a stochastic recursive gradient descent ascent (SREDA) algorithm, and the gradient complexity of which is $\tilde{\mathcal{O}}( \kappa^{3}\varepsilon^{-3} )$. For general nonconvex-concave minimax problem, many nested-loop algorithms have been proposed in \citep{rafique2018non,nouiehed2019solving,thekumparampil2019efficient,kong2019accelerated,ostrovskii2020efficient,yang2020catalyst}.  To the best of our knowledge, \citet{lin2020near} proposed an accelerated algorithms called MINIMAX-PPA, which has the best iteration complexity of $\tilde{\mathcal{O}}( \varepsilon ^{-2.5} )$ till now. Several single-loop methods were also proposed to solve the problem. GDA-type algorithms~\citep{Chambolle,Ho2016,Daskalakis17,Daskalakis18,gidel2018variational,Letcher,lin2020gradient,lu2020hybrid,pan2021an,shen2022zeroth,zhang2020single}, which run a gradient descent step on $x$ and a gradient ascent step on $y$ simultaneously at each iteration.  \citet{xu2020unified} proposed a unified single-loop alternating gradient projection (AGP) algorithm for solving nonconvex-(strongly) concave and (strongly) convex-nonconcave minimax problems, which can find an $\varepsilon$-stationary point with the gradient complexity of $\mathcal{O}( \varepsilon ^{-4} )$. 

Now we give a brief introduction about variance reduced algorithms for minimax optimization. The variance reduced technique was first proposed for sloving general minimization optimization problem, and many classical algoritms including SAGA, SVRG, SARAH, SPIDER and SpiderBoost all employ variance reduced technique \citep{defazio2014saga,reddi2016b,johnson2013accelerating,allen2016variance,allen2017natasha,nguyen2017sarah,nguyen2017stochastic,nguyen2018inexact,fang2018spider,wang2019spiderboost}. For nonconvex-strongly concave minimax optimization, several variance reduction methods have been proposed for solving minimax optimization, such as  PGSVRG \citep{rafique2018non}, the SAGA-type algorithm  \citep{wai2019variance}, and SREDA \citep{luo2020stochastic}. However, to the best of our knowledge, there are no zeroth-order algorithms utilizing variance reduction techniques for solving general nonconvex-concave minimax optimization problems.

In this paper, we consider a class of nonconvex-nonconcave minimax problems, i.e., the nonconvex-PL (NC-PL) minimax problems, for which we assume $f(x,y)$ in \eqref{min_max_problem_deterministic} and \eqref{min_max_problem_stochastic} satisfies the Polyak-\L ojasiewicz (PL) condition with respect to $y$, which is the same as in \citep{nouiehed2019solving,yang2020catalyst,yang2020global}. This condition was originally introduced by \citep{polyak1963gradient} and is proved to be weaker than strong convexity in~\citep{karimi2016linear}. The PL condition has also drawn much attention in machine learning and deep learning problems, and has been shown to hold in linear quadratic regulators \citep{fazel2018global}, as well as overparametrized neural networks~\citep{liu2020loss}. For NC-PL minimax problems, we propose a zeroth-order alternating gradient descent ascent (ZO-AGDA) algorithm and a zeroth-order variance reduced alternating gradient descent ascent (ZO-VRAGDA) algorithm  for solving  \eqref{min_max_problem_deterministic} and \eqref{min_max_problem_stochastic} with the total number of function value queries of  $\mathcal{O}(\varepsilon^{-2})$ and $\mathcal{O}(\varepsilon^{-3})$ respectively. To the best of our knowledge, they are the first two zeroth-order algorithms with the complexity gurantee for solving NC-PL minimax problems under both the deterministic and the stochastic setting. 


\section{Preliminaries}\label{section preliminaries}
\subsection{Notations}
Throughout the paper, we use the following notations. $\langle x,y \rangle$ denotes the inner product of two vectors of $x$ and $y$. $\| \cdot \|$ denotes the Euclidean norm. 
We use $\mathbb{R}^{d_{1}}$ to denote the space of $d_{1}$ dimension real valued vectors. 
$\mathbb{E}_u (\cdot)$ means the expectation over the random vector $u$, $\mathbb{E}_{(u, \xi)}(\cdot) $ means the joint expectation over the random vector $u$ and the random variable $\xi$, and  $\mathbb{E}_{(U, \mathcal{B})}$ denotes the joint expectation over the set $U$ of random vectors and the set $\mathcal{B}$ of random variables $\{\xi_{1}, \cdots, \xi_{b}\}$. Denote $\Phi(x) = \max\limits_{y} f(x,y)$, $\Phi^*=\min\limits_{x} \Phi(x)$ and $\Psi(x) = \max\limits_{y} g(x,y)$, $\Psi^*=\min\limits_{x} \Psi(x)$. 

\subsection{Zeroth-order gradient estimator}\label{UniGE}
For solving problems \eqref{min_max_problem_deterministic} and \eqref{min_max_problem_stochastic}, since the gradient information is not available directly,  we first introduce the idea of uniform smoothing gradient estimator (UniGE). Specifically, for \eqref{min_max_problem_deterministic}, the UniGE of $\nabla_{x} f(x,y)$ and $\nabla_{y} f(x,y)$ are respectively defined as
\begin{align}
	&\hat{\nabla}_{x} f(x, y)= \frac{f(x+\mu_{1} {u}, y)-f(x, y)}{\mu_{1} / d_{1}} {u}, \label{sec2:1}\\
	&\hat{\nabla}_{y} f(x, y)= \frac{f(x, y+\mu_{2} {v})-f(x, y)}{\mu_{2} / d_{2}} {v}\label{sec2:2},
\end{align}
where $\mu_{1},\mu_{2}$ are two smoothing parameters, $u\in \mathbb{R}^{d_{1}}$ and $v \in \mathbb{R}^{d_{2}}$ are random vectors that are generated from the uniform distribution over $d_{1}$-dimensional and $d_{2}$-dimensional unit sphere respectively. For convenience, for any $u\in  \mathbb{R}^{d_{1}}$ and $v\in  \mathbb{R}^{d_{2}}$, we denote
\begin{align*}
	f_{\mu_{1}}(x, y)&=\mathbb{E}_{u} f(x+\mu_{1} u, y),\\
	f_{\mu_{2}}(x, y)&=\mathbb{E}_{v} f(x, y+\mu_{2} v).
\end{align*}
Note that $\mathbb{E}_u (\hat{\nabla}_{x} f(x, y))= \nabla_{x} f_{\mu_{1}}(x, y)$ and $\mathbb{E}_v (\hat{\nabla}_{y} f(x, y))= \nabla_{y} f_{\mu_{2}}(x, y)$ by Lemma 5 in \citep{ji2019improved}.

Similarly, for \eqref{min_max_problem_stochastic}, given $\mathcal{B}=\{\xi_{1}, \cdots, \xi_{r}\}$ and $\bar{\mathcal{B}}=\{\zeta_{1}, \cdots, \zeta_{r}\}$ drawn i.i.d. from an unknown distribution $\mathcal{P}$, we respectively  define the UniGE of $\nabla_{x} G(x,y,\xi)$ and $\nabla_{y} G(x,y,\zeta)$ as
\begin{align*}
	&\hat{\nabla}_{x} G(x, y ; \mathcal{B})=\frac{1}{r} \sum_{i=1}^{r} \hat{\nabla}_{x} G(x, y ; \xi_{i})=\frac{1}{r} \sum_{i=1}^{r} \frac{G(x+\mu_{1} {u}_{i}, y ; \xi_{i})-G(x, y ; \xi_{i})}{\mu_{1} / d_{1}} {u}_{i}, \\
	&\hat{\nabla}_{y} G(x, y ; \bar{\mathcal{B}})=\frac{1}{r} \sum_{i=1}^{r} \hat{\nabla}_{y} G(x, y ; \zeta_{i})=\frac{1}{r} \sum_{i=1}^{r} \frac{G(x, y+\mu_{2} {v}_{i} ; \zeta_{i})-G(x, y ; \zeta_{i})}{\mu_{2} / d_{2}} {v}_{i},
\end{align*}
where $u_{i} \in \mathbb{R}^{d_{1}}$ and $v_{i} \in \mathbb{R}^{d_{2}}$ are random vectors generated from the uniform distribution over $d_{1}$-dimensional and $d_{2}$-dimensional unit sphere respectively. We also denote that
\begin{align*}
	g_{\mu_{1}}(x, y)&=\mathbb{E}_{(u, \xi)} G(x+\mu_{1} u, y;\xi),\\
	g_{\mu_{2}}(x, y)&=\mathbb{E}_{(v, \zeta)} G(x, y+\mu_{2} v;\zeta)
\end{align*}
for any $u\in \mathbb{R}^{d_{1}}$ and $v\in \mathbb{R}^{d_{2}}$.

Note that for any random variable $\xi$, we have $\mathbb{E}_{(u, \xi)}[\hat{\nabla}_{x} G(x, y ; \xi)]=\nabla_{x} g_{\mu_{1}}(x, y)$ and $\mathbb{E}_{(v, \zeta)}[\hat{\nabla}_{y} G(x, y ; \zeta)]=\nabla_{y} g_{\mu_{2}}(x, y)$ by Lemma 5 in \citep{ji2019improved}. Then obviously we have $\mathbb{E}_{(U, \mathcal{B})}[\hat{\nabla}_{x} G(x, y ; \mathcal{B})]=\nabla_{x} g_{\mu_{1}}(x, y)$ and $ \mathbb{E}_{(V, \bar{\mathcal{B}})}[\hat{\nabla}_{y} G(x, y ; \bar{\mathcal{B}})]=\nabla_{y} g_{\mu_{2}}(x, y)$ where $U=\{u_1,\cdots, u_r\}$ and $V=\{v_1, \cdots, v_r\}$.

Actually, there are two major zeroth-order gradient estimators that are usually used in previous exisiting works. One is the UniGE, which has also been used in some other existing works, e.g., \citep{liu2020min,bravo2018bandit,ji2019improved,huang2020accelerated,xu2021zeroth}. Another commonly used zeroth-order gradient estimator is Gaussian gradient estimator~\citep{ghadimi2016mini,yurii17random,fazel2018global}, which can not lead to a better iteration complexity than that of UniGE when used in the two proposed algorithms that shown in the following sections.

\section{A Zeroth-order Algorithm for Deterministic NC-PL Minimax Problems}\label{section ZO-AGDA}
The alternating gradient descent ascent (AGDA) algorithm is a  well-known method for solving minimax problems, and has widely been studied in \citep{Letcher,Chambolle,Daskalakis17,Daskalakis18,lin2020gradient}. ZO-AGDA algorithm is a randomized version of AGDA algorithm, and is not a new algorithm, which has been proposed in \citep{liu2020min,wang2020zerothorder}. However, the two existing zeroth-order versions of the AGDA algorithm are designed and analyzed for solving nonconvex-strongly concave minimax problems. For NC-PL minimax problems, there is no existing algorithms with complexity guarantee before. In this section, we propose a zeroth-order alternating gradient descent ascent (ZO-AGDA) algorithm for solving \eqref{min_max_problem_deterministic}, i.e., the deterministic NC-PL problem, and analyze its iteration complexity. The detailed algorithm is listed as follows.
\begin{algorithm}	
	\caption{(ZO-AGDA Algorithm)}\label{algo1}
	\begin{algorithmic}
		\item [Step 1] ~ Input $x_1,y_1,\alpha,\beta$; Set $t=1$.\\	
		\item [Step 2] ~ Perform the following update for $x_t$:
		\qquad 	\begin{equation}\label{algolstep_x}
			x_{t+1}=x_{t}-\alpha \hat{\nabla}_{x} f(x_{t}, y_{t})
		\end{equation}
		\qquad \quad with $\hat{\nabla}_{x} f(x_{t}, y_{t})$ being defined as in \eqref{sec2:1};
		\item [Step 3] ~ Perform the following update for $y_t$:  	
		\qquad  \begin{equation}\label{algolstep_y}
			y_{t+1}=y_{t}+\beta \hat{\nabla}_{y} f(x_{t+1}, y_{t})
		\end{equation}
		\qquad \quad with $\hat{\nabla}_{y} f(x_{t+1}, y_{t})$ being defined as in \eqref{sec2:2};
		\item [Step 4] ~ If converges, stop; otherwise, set $t=t+1, $ go to Step 2.
	\end{algorithmic}
\end{algorithm}
For simplicity, in the following analysis, we denote $s_{t}=\hat{\nabla}_{x} f(x_{t}, y_{t})$ and $w_{t}=\hat{\nabla}_{y} f(x_{t+1}, y_{t})$.

\subsection{Technical Preparations}

In this section, we analyze the iteration complexity of the ZO-AGDA algorithm for solving \eqref{min_max_problem_deterministic}.

Firstly, we give some mild assumptions for \eqref{min_max_problem_deterministic}.

\begin{assumption}\label{a3} 
	For any fixed $x, \max\limits_{y} f(x, y)$ has a nonempty solution set and a finite optimal value. $ f(x,y)$ satisfies the Polyak-\L ojasiewicz (PL) condition in $y$, i.e., $\forall x, y$, there exists a $\mu>0$ such that $\|\nabla_{y} f(x, y)\|^{2} \geqslant 2\mu$ $[\max\limits_{y} f(x, y)-f(x, y)]$. 
\end{assumption}
\begin{assumption}\label{a1}
	$f(x,y)$ has Lipschitz continuous gradients, i.e., there exists a constant $l>0$ such that $\forall x_{1}, x_{2} \in \mathbb{R}^{d_{1}}$, $y_{1}, y_{2} \in \mathbb{R}^{d_{2}}$,
	\begin{align*}
		\|\nabla_x f(x_{1},y_{1}) - \nabla_x f(x_{2},y_{2})\| \leqslant l[\|x_{1}-x_{2}\|+\|y_{1}-y_{2}\|],\\
		\|\nabla_y f(x_{1},y_{1}) - \nabla_y f(x_{2},y_{2})\| \leqslant l[\|x_{1}-x_{2}\|+\|y_{1}-y_{2}\|].
	\end{align*}
\end{assumption}
\begin{lemma} (Lemma 4.1(a) in \citep{gao2018information})\label{zoagda_lemma}
	\label{funconvexlip}
	If Assumption \ref{a1} holds, then $f_{\mu_{1}}(x,y)$ and $f_{\mu_{2}}(x,y)$ have Lipschitz continuous gradients, and it holds that
	\begin{align}
		&\|\nabla_{x}f_{\mu_{1}}(x,y) - \nabla_{x}f(x,y)
		\|^{2} \leqslant \frac{\mu_{1}^{2}d_{1}^{2}l^{2}}{4},\label{gradgap_x}\\ &
		\|\nabla_{y}f_{\mu_{2}}(x,y) - \nabla_{y}f(x,y)\|^{2} \leqslant \frac{\mu_{2}^{2}d_{2}^{2}l^{2}}{4},\label{gradgap}\\
		&\mathbb{E}\|\frac{d_{1}[f(x+\mu_{1}u,y) - f(x,y)]}{\mu_{1}} u\|^{2} \leqslant 2d_{1}\|\nabla_{x}f(x,y)\|^{2} + \frac{\mu_{1}^{2}d_{1}^{2}l^{2}}{2},\label{gradboundx}\\
		&\mathbb{E}\|\frac{d_{2}[f(x,y+\mu_{2}v) - f(x,y)]}{\mu_{2}} v\|^{2} \leqslant 2d_{2}\|\nabla_{y}f(x,y)\|^{2} + \frac{\mu_{2}^{2}d_{2}^{2}l^{2}}{2}.\label{gradboundy}
	\end{align}
\end{lemma}

\begin{lemma} (Lemma A.5 in \citep{nouiehed2019solving}) \label{philip}
	If Assumptions \ref{a3} and \ref{a1} hold, then  $\nabla\Phi(x) = \nabla_{x}f(x,y^{*}(x))$, for any $y^{*}(x) \in \underset{y}{\arg\max} f(x,y)$, and $\Phi(\cdot)$ is $L$-smooth with $L:=l+\frac{l^{2}}{2\mu}$.
\end{lemma}

\begin{lemma}\label{g smooth}
	If Assumptions \ref{a3} and \ref{a1} hold, then for any given $x$ and $y$, we have 
	\begin{align}\label{phigap}
		\|\nabla_{x} f(x, y)-\nabla \Phi(x)\|^{2} \leqslant \kappa^{2}\|\nabla_{y} f(x, y)\|^{2},
	\end{align}
	where $\kappa=l/\mu$.
\end{lemma}
\begin{proof}
	If Assumptions \ref{a3} and \ref{a1} hold, then by Theorem 1 in \citep{karimi2016linear}, for any given $x$, and $\forall y$,
	\begin{equation}
		\Vert \nabla_{y} f(x,y)\Vert \geqslant \mu \Vert \hat{y}^{*}(x,y)-y\Vert, \label{add:1}
	\end{equation}
	where $\hat{y}^{*}(x,y)=\arg\min \{\|y-y^{*}(x)\| \mid y^{*}(x) \in \underset{y}{\arg\max} f(x,y)\}$. Then, by Assumption \ref{a1} and \eqref{add:1}, we obtain
	\begin{align*}
		\|\nabla_{x} f(x, y)-\nabla \Phi(x)\|^{2} \leqslant l^{2}\|y-\hat{y}^{*}(x, y)\|^{2} \leqslant \kappa^{2}\|\nabla_{y} f(x, y)\|^{2}.
	\end{align*}
\end{proof}

	%
	%
%

\subsection{Complexity Analysis}\label{section Analysis}
We first give the $\epsilon$-stationary point definition of \eqref{min_max_problem_deterministic} as follows, which is also used in \citep{yang22b}. 
\begin{definition} $\hat{x}$ is an $\epsilon$-stationary point of \eqref{min_max_problem_deterministic} if $\mathbb{E} \|\nabla\Phi(\hat{x})\|\leqslant\epsilon$.
\end{definition}

\begin{lemma}
	Denote $V_{t}=\frac{3}{2} \Phi(x_{t})-\frac{1}{2} f(x_{t}, y_{t})$. If Assumptions \ref{a3} and \ref{a1} hold, and if $\beta\leqslant\frac{1}{4d_{2}L}$, $\alpha\leqslant\text{min}\{\frac{\beta}{32\kappa^{2}},\frac{1}{10d_{1}L}\}$, we get
	\begin{align} 
		\mathbb{E}V_{t}-\mathbb{E}V_{t+1}&\geqslant
		\frac{\alpha}{4}\mathbb{E}\|\nabla\Phi(x_{t})\|^{2}
		-\frac{\theta_{1}}{4}\mu_{1}^{2}
		-\frac{3d_{2}^{2} L^{2}\beta}{16}\mu_{2}^{2},
		\label{the1}
	\end{align}
	where $\theta_{1}=(5d_{1}L+\frac{3}{2\alpha}+\frac{3}{2}L+d_{2}L)d_{1}^{2}L^{2}\alpha^{2}$.
\end{lemma}

\begin{proof}
	By Assumption \ref{a1} and \eqref{algolstep_y},  we can get		
	\begin{align} \nonumber
		f(x_{t+1}, y_{t+1}) & \geqslant f(x_{t+1}, y_{t})+\langle\nabla_{y} f(x_{t+1}, y_{t}), y_{t+1}-y_{t}\rangle-\frac{l}{2}\|y_{t+1}-y_{t}\|^{2} \\
		&\nonumber = f(x_{t+1}, y_{t})+\langle\nabla_{y} f(x_{t+1}, y_{t}), \beta\omega_{t}\rangle-\frac{l}{2}\beta^{2}\|\omega_{t}\|^{2} .
	\end{align}		
	Similarly, by Assumption \ref{a1} and \eqref{algolstep_x}, we have
	\begin{align} \nonumber
		f(x_{t+1}, y_{t}) & \geq f(x_{t}, y_{t})+\langle\nabla_{x} f(x_{t}, y_{t}), x_{t+1}-x_{t}\rangle-\frac{l}{2}\|x_{t+1}-x_{t}\|^{2} \\
		&\nonumber = f(x_{t}, y_{t})-\langle\nabla_{x} f(x_{t}, y_{t}),\alpha s_{t} \rangle-\frac{l}{2}\alpha^{2}\|s_{t}\|^{2}.
	\end{align}
	Taking expectation of the above two inequalities, we get
	\begin{align}
		&\quad\mathbb{E} f(x_{t+1}, y_{t+1})-\mathbb{E} f(x_{t+1}, y_{t})\nonumber\\
		&\geqslant \beta \mathbb{E} \langle\nabla_{y} f(x_{t+1}, y_{t}), \nabla_{y} f_{u}(x_{t+1}, y_{t})\rangle-\frac{l}{2} \beta^{2} \mathbb{E} \|\omega_{t}\|^{2}, \label{sec3:1}
	\end{align}
	and
	\begin{align}
		&\quad\mathbb{E} f(x_{t+1}, y_{t})-\mathbb{E} f(x_{t}, y_{t})\nonumber \\
		&\geqslant -\alpha\mathbb{E}\langle\nabla_{x} f(x_{t}, y_{t}),  \nabla_{x} f_{u}(x_{t}, y_{t})\rangle-\frac{l}{2} \alpha^{2} \mathbb{E} \| s_{t}\|^{2}.\label{sec3:2}
	\end{align}
	Combining \eqref{sec3:1} and \eqref{sec3:2}, we obtain
	\begin{align} \nonumber
		&\quad\mathbb{E}f(x_{t+1}, y_{t+1})-\mathbb{E} f(x_{t}, y_{t})\\
		&\geqslant \nonumber
		\beta \mathbb{E}\langle\nabla_{y} f(x_{t+1}, y_{t}), \nabla_{y} f(x_{t}, y_{t})\rangle-\frac{l}{2} \beta^{2} \mathbb{E}\|w_{t}\|^{2} \\ \nonumber
		&\quad-\alpha \mathbb{E} \langle\nabla_{x} f(x_{t}, y_{t}), \nabla_{x} f_{\mu}(x_{t}, y_{t})\rangle-\frac{l}{2} \alpha^{2} \mathbb{E}\|s_{t}\|^{2} \\ \nonumber
		&=\frac{\beta}{2} \mathbb{E}\|\nabla_{y} f(x_{t+1} y_{t})\|^{2}-\frac{\beta}{2} \mathbb{E}\|\nabla_{y} f(x_{t+1,} y_{t})-\nabla_{y} f_{\mu}(x_{t+1}, y_{t})\|^{2} \\ \nonumber
		&\quad+\frac{\beta}{2} \mathbb{E}\|\nabla_{y} f_{\mu}(x_{t+1}, y_{t})\|^{2}-\frac{l}{2} \beta^{2} \mathbb{E}\|w_{t}\|^{2}+\frac{\alpha}{2} \mathbb{E}\|\nabla_{x} f(x_{t}, y_{t})-\nabla_{x} f_{\mu}(x_{t}, y_{t})\|^{2} \\ 
		&\quad-\frac{\alpha}{2} \mathbb{E}\|\nabla_{x} f(x_{t}, y_{t})\|^{2}-\frac{\alpha}{2} \mathbb{E}\|\nabla_{x} f_{\mu}(x_{t}, y_{t})\|^{2}-\frac{l}{2} \alpha^{2} \mathbb{E}\|s_{t}\|^{2},
		\label{sec3:3}
	\end{align}
	where the last equality is due to a simple fact that  $ 2\langle a,b \rangle=\|a+b\|^{2}-\|a\|^{2}-\|b\|^{2}$.
	By $L$-smoothness of $\Phi$ in Lemma \ref{philip} and \eqref{algolstep_x}, we get 
	\begin{align*}
		\Phi(x_{t+1}) & \leqslant \Phi(x_{t})+\langle\nabla \Phi(x_{t}), x_{t+1}-x_{t}\rangle+\frac{L}{2}\|x_{t+1}-x_{t}\|^{2} \\
		&=\Phi(x_{t})-\alpha\langle\nabla \Phi(x_{t}), s_{t}\rangle+\frac{L}{2} \alpha^{2}\|s_{t}\|^{2}.
	\end{align*}
	Taking expectation of both side and using the fact that $ 2\langle a,b \rangle=\|a+b\|^{2}-\|a\|^{2}-\|b\|^{2}$,  we get
	\begin{align} \nonumber
		&\mathbb{E}\Phi(x_{t+1})-\mathbb{E}\Phi(x_{t})\\
		& \leqslant -\alpha\mathbb{E}\langle\nabla \Phi(x_{t}), \nabla_{x}f_{\mu}(x_{t},y_{t})\rangle+\frac{L}{2}\alpha^{2}\mathbb{E}\|s_{t}\|^{2}\nonumber\\
		&=
		\frac{\alpha}{2}\mathbb{E}\|\nabla\Phi(x_{t})
		-\nabla_{x}f_{\mu}(x_{t},y_{t})\|^{2}\\ \nonumber
		&\quad-\frac{\alpha}{2}\mathbb{E}\|\nabla\Phi(x_{t})\|^{2}
		-\frac{\alpha}{2}\mathbb{E}\|\nabla_{x}f_{\mu}(x_{t},y_{t})\|^{2}
		+\frac{L}{2}\alpha^{2}\mathbb{E}\|s_{t}\|^{2}\\ \nonumber
		&\leqslant \alpha\mathbb{E}\|\nabla\Phi(x_{t})-\nabla_{x}f(x_{t},y_{t})\|^{2} +\alpha\mathbb{E}\|\nabla_{x}f(x_{t},y_{t})-\nabla_{x}f_{\mu}(x_{t},y_{t})\|^{2}\\ \nonumber
		& \quad-\frac{\alpha}{2}\mathbb{E}\|\nabla\Phi(x_{t})\|^{2}-\frac{\alpha}{2}\mathbb{E}\|\nabla_{x}f_{\mu}(x_{t},y_{t})\|^{2}+\frac{L}{2}\alpha^{2}\mathbb{E}\|s_{t}\|^{2}\\ \nonumber
		&\leqslant \alpha\kappa^{2}\|\nabla_{y}f(x_{t},y_{t})\|^{2}+\alpha\mathbb{E}\|\nabla_{x}f(x_{t},y_{t})-\nabla_{x}f_{\mu}(x_{t},y_{t})\|^{2}\\
		&\quad-\frac{\alpha}{2}\mathbb{E}\|\nabla\Phi(x_{t})\|^{2}-\frac{\alpha}{2}\mathbb{E}\|\nabla_{x}f_{\mu}(x_{t},y_{t})\|^{2}+\frac{L}{2}\alpha^{2}\mathbb{E}\|s_{t}\|^{2},
		\label{sec3:4}
	\end{align}
	where the second last inequailty is by the Cauchy-Schwarz inequality and the last inequality is by Lemma \ref{g smooth}.
	Combining  \eqref{sec3:3} and \eqref{sec3:4} and by the definition of $V_t$, after rearranging the terms, we can get that 
	\begin{align} \nonumber
		&\mathbb{E}V_{t}-\mathbb{E}V_{t+1}\\
		\nonumber\geqslant& \frac{3\alpha}{4}\mathbb{E}\|\nabla\Phi(x_{t})\|^{2}
		+\frac{3\alpha}{4}\mathbb{E}\|\nabla_{x}f_{\mu}(x_{t},y_{t})\|^{2} \\
		\nonumber
		&-\frac{3L}{4}\alpha^{2}\mathbb{E}\|s_{t}\|^{2}
		-\frac{3\alpha}{2}\kappa^{2}\mathbb{E}\|\nabla_{y}f(x_{t},y_{t})\|^{2}\\
		\nonumber
		&-\frac{3\alpha}{2}\mathbb{E}\|\nabla_{x}f(x_{t},y_{t})-\nabla_{x}f_{\mu}(x_{t},y_{t})\|^{2}+\frac{\beta}{4}\mathbb{E}\|\nabla_{y}f(x_{t+1},y_{t})\|^{2}\\
		\nonumber			
		&-\frac{l^{2}}{4}\beta^{2}\mathbb{E}\|\omega_{t}\|^{2}-\frac{\alpha}{4}\mathbb{E}\|\nabla_{x}f_{\mu}(x_{t},y_{t})\|^{2}-\frac{l}{4}\alpha^{2}\mathbb{E}\|s_{t}\|^{2}\\
		\nonumber
		&-\frac{\beta}{4}\mathbb{E}\|\nabla_{y}f(x_{t+1},y_{t})-\nabla_{y}f_{\mu}(x_{t+1},y_{t})\|^{2}
		-\frac{\alpha}{4}\mathbb{E}\|\nabla_{x}f(x_{t},y_{t})\|^{2}\\
		\nonumber
		\geqslant&\frac{3\alpha}{4}\mathbb{E}\|\nabla\Phi(x_{t})\|^{2}+\frac{\alpha}{2}\mathbb{E}\|\nabla_{x}f_{\mu}(x_{t},y_{t})\|^{2}-L\alpha^{2}\mathbb{E}\|s_{t}\|^{2}\\
		\nonumber
		&-\frac{3\alpha}{2}\kappa^{2}\mathbb{E}\|\nabla_{y}f(x_{t},y_{t})\|^{2}-\frac{3\alpha}{2}\mathbb{E}\|\nabla_{x}f(x_{t},y_{t})-\nabla_{x}f_{\mu}(x_{t},y_{t})\|^{2}
		\\
		\nonumber
		&+(\frac{\beta}{4}-\frac{\beta^{2}}{2}d_{2}L)\mathbb{E}\|\nabla_{y}f(x_{t+1},y_{t})\|^{2}-\frac{\beta}{4}\mathbb{E}\|\nabla_{y}f(x_{t+1},y_{t})-\nabla_{y}f_{\mu}(x_{t+1},y_{t})\|^{2}\\
		&-\frac{\alpha}{4}\mathbb{E}\|\nabla_{x}f(x_{t},y_{t})\|^{2}-\frac{L^{3}}{8}\beta^{2}\mu_{2}^{2}d_{2}^{2},\label{sec3:5}
	\end{align}
	where the second inequality is due to that $\frac{\beta}{4}-\frac{\beta^{2}}{2}d_{2}L>0$ when $\beta\leqslant\frac{1}{4d_{2}L}$ and by replacing $l$ with $L$ since $l \leqslant L$ and \eqref{gradboundy} in Lemma \ref{zoagda_lemma}.
	On the other hand, by the Cauchy-Schwarz inequality and \eqref{phigap} in Lemma \ref{g smooth}, we have
	\begin{align}
		\|\nx f(\xt,\yt)\|^{2}
		& \nonumber
		\leqslant 2\|\nx f(\xt,\yt)-\np(\xt)\rl+2\|\np(\xt)\|^{2}\\
		& 
		\leqslant 2\kappa^{2}\|\ny f(\xt,\yt)\rl+2\|\np(\xt)\rl.
		\label{sec3:6}
	\end{align}
	By using a simple inequality that $\|a\|^{2} \geqslant\|b\|^{2} / 2-\|a-b\|^{2}$, Assumption \ref{a1} and the definition of $s_t$, we can easily get  
	\begin{align}
		\|\ny f(\xtt,\yt)\rl & \nonumber
		\geqslant \frac{1}{2}\|\ny f(\xt,\yt)\rl-\|\ny f(\xtt,\yt)-\ny f(\xt,\yt)\rl\\
		&
		\geqslant \frac{1}{2}\|\ny f(\xt,\yt)\rl-l^{2}\alpha^{2}\|\vt\rl.
		\label{sec3:7}
	\end{align}
	Denote $G_{1}=L\alpha^{2}+(\frac{\beta}{4}-\frac{\beta^{2}}{2}d_{2}L) L^{2}\alpha^{2}$ and $G_{2}=\frac{\beta}{8}-\frac{L}{4}d_{2}\beta^{2}-2\alpha\kappa^{2}$. 
	By plugging \eqref{sec3:6} and \eqref{sec3:7} into \eqref{sec3:5} and rearranging all the terms, we obtain
	\begin{align} 
		&\quad\mathbb{E}V_{t}-\mathbb{E}V_{t+1} \nonumber\\
		& \geqslant \nonumber
		\frac{\alpha}{4}\mathbb{E}\|\nabla\Phi(x_{t})\|^{2}+\fa\BE \|\nx\fu(\xt,\yt)\rl-G_{1}\mathbb{E}\|s_{t}\|^{2}\\
		& \nonumber
		\quad+G_{2}\mathbb{E}\|\nabla_{y}f(x_{t},y_{t})\|^{2}
		-\frac{3\alpha}{2}\mathbb{E}\|\nabla_{x}f(x_{t},y_{t})-\nabla_{x}f_{\mu}(x_{t},y_{t})\|^{2}\\
		&
		\quad-\frac{\beta}{4}\mathbb{E}\|\nabla_{y}f(x_{t+1},y_{t})-\nabla_{y}f_{\mu}(x_{t+1},y_{t})\|^{2}-\frac{L^{3}}{8}\beta^{2}\mu_{2}^{2}d_{2}^{2}.\label{add:2}
	\end{align}
	By \eqref{gradboundx} in Lemma \ref{zoagda_lemma}, we can compute that
	\begin{align}
		\mathbb{E}\|s_{t}\|^{2}&\leqslant4d_{1}\mathbb{E}\|\nabla_{x}f(x_t,y_t)-\nabla_{x}f_{\mu}(x_t,y_t)\|^{2}+4d_{1}\mathbb{E}\|\nabla_{x}f_{\mu}(x_t,y_t)\|^{2}+\frac{\mu_{1}^{2}d_{1}^{2}l^{2}}{2}.\label{add:3}
	\end{align} 
	By plugging \eqref{add:3} into \eqref{add:2},	
	\begin{align} \nonumber
		&\quad \mathbb{E}V_{t}-\mathbb{E}V_{t+1} \\
		\geqslant &  \nonumber
		\frac{\alpha}{4}\mathbb{E}\|\nabla\Phi(x_{t})\|^{2}
		+(\fa-4d_{1}G_{1})\BE \|\nx\fu(\xt,\yt)\rl\\
		& \nonumber
		+G_{2}\mathbb{E}\|\nabla_{y}f(x_{t},y_{t})\|^{2}-(\frac{3\alpha}{2}+4d_{1}G_{1})
		\mathbb{E}\|\nabla_{x}f(x_{t},y_{t})-\nabla_{x}f_{\mu}(x_{t},y_{t})\|^{2}\\
		&
		-\frac{\beta}{4}\mathbb{E}\|\nabla_{y}f(x_{t+1},y_{t})-\nabla_{y}f_{\mu}(x_{t+1},y_{t})\|^{2}
		-\frac{\mu_{1}^{2}}{2}G_{1}d_{1}^{2}l^{2}
		-\frac{L^{3}}{8}\beta^{2}\mu_{2}^{2}d_{2}^{2}.\label{add:4}
	\end{align} 
	When $\beta\leqslant\frac{1}{4d_{2}L}\leqslant\frac{1}{L}$ and $\alpha\leqslant\text{min}\{\frac{\beta}{32\kappa^{2}},\frac{1}{10d_{1}L}\}$, 
	it can be easily checked that $\frac{\alpha}{2}-4d_{1}G_{1}\geqslant 0$, $\frac{3\alpha}{2}+4d_{1}G_{1}\leqslant \frac{3}{2}\alpha+4d_{1}\alpha^{2}L+d_{1}\beta L^{2}\alpha^{2}$, and $G_{2}=\frac{\beta}{8}(1-2\beta Ld_{2})-2\alpha\kappa^{2}\geqslant\frac{\beta}{16}-2\alpha\kappa^{2}\geqslant 0$.
	The proof is then completed by combining \eqref{add:4} and \eqref{gradgap} in Lemma \ref{zoagda_lemma}.
\end{proof}
Denote $T(\epsilon):=\min\{t\ | \ \mathbb{E}\|\nabla \Phi(x_{t})\|\leqslant\epsilon \}$, which means the minimal number of  iterations to obtain an $\epsilon$-stationary point.
\begin{theorem}\label{zoagda_theorem}
	Suppose that Assumptions \ref{a3} and \ref{a1} hold, and $\Phi^{*}$ exists and is a finite value. Let $\mu_{1}=\frac{\sqrt{\alpha}\epsilon}{2\sqrt{\theta_{1}}},\mu_{2}=\frac{\sqrt{\alpha}\epsilon}{\sqrt{3\beta}d_{2}L}$. If $\alpha\leqslant\min\{\frac{\beta}{32\kappa^{2}},\frac{1}{10d_{1}L}\}$ and $\beta\leqslant\frac{1}{4d_{2}L}$, we have
	\begin{align*}
		T(\varepsilon)&\leqslant \frac{4[3\Phi(x_{0})-f(x_0, y_0)-2\Phi^{*}]}{\alpha\epsilon^{2}}.\end{align*}
\end{theorem}

\begin{proof}
	Telescoping and rearranging \eqref{the1}, we get 
	\begin{align*}
		&\sum_{t=0}^{T(\epsilon)-1} \mathbb{E}\|\nabla \Phi(x_{t})\|^{2} \\
		&\leqslant \frac{4}{\alpha}\left[V_{0}-\min _{x, y} \left(\frac{3}{2} \Phi(x)-\frac{1}{2} f(x, y)\right)\right]+\frac{\theta_{1}}{\alpha}\mu_{1}^{2}T(\epsilon)+\frac{3d_{2}^{2} L^{2}\beta}{4\alpha}\mu_{2}^{2}T(\epsilon).
	\end{align*}
	By the definitions of $T(\epsilon)$, $\mu_{1}$ and $\mu_{2}$, we have
	\begin{align}
		\epsilon^{2}&\nonumber \leqslant \frac{4}{\alpha T(\epsilon)}\left[V_{0}-\min _{x, y} \left(\frac{3}{2} \Phi(x)-\frac{1}{2} f(x, y)\right)\right]+\frac{\epsilon^{2}}{2}\\
		& \leqslant \frac{4}{\alpha T(\epsilon)}[V_{0}-\Phi^{*}]+\frac{\epsilon^{2}}{2}\\
		& = \frac{2}{\alpha T(\epsilon)}[3\Phi(x_{0})-f(x_0, y_0)-2\Phi^{*}]+\frac{\epsilon^{2}}{2},
		\label{theorem1}
	\end{align}
	where  the second inequality is due to $\Phi(x)\geqslant f(x,y)$ by the definition of $\Phi(x)$. 
	The proof is completed by \eqref{theorem1}.
\end{proof}

By choosing $\beta=\frac{1}{4d_{2}L}$ and  $\alpha=\min\{\frac{\beta}{32\kappa^{2}},\frac{1}{10d_{1}L}\}$ in Theorem \ref{zoagda_theorem}, we can easily compute that $\mu_{1}=\mathcal{O}(\sqrt{\frac{\kappa^2 d_2+ d_1}{d_2+d_1}}\cdot \frac{\epsilon}{d_{1} L})$, $\mu_{2}=\mathcal{O}(\frac{\mu }{1+ \mu \sqrt{d_1d_2}}\cdot \frac{\epsilon}{d_{2} L^2})$ and   $T(\epsilon)=\mathcal{O}(( \kappa^{2} d_2+ d_1)L \epsilon^{-2})$, which means the iteration complexity of Algorithm \ref{algo1} to find an $\epsilon$-stationary point for \eqref{min_max_problem_deterministic} is  $\mathcal{O}(\epsilon^{-2})$.
Hence, by \eqref{sec2:1} and \eqref{sec2:2}, the total number of function value queries of  Algorithm \ref{algo1} is $4*T(\epsilon)$ which is the same order of  $\mathcal{O}(\epsilon^{-2})$.

\section{A Zeroth-order Algorithm for Stochastic NC-PL Minimax Problems}
In this section, we propose a new zeroth-order gradient descent ascent method with variance reduction (ZO-VRAGDA) for solving \eqref{min_max_problem_stochastic}, under the setting that only noisy function values can be used. At each iteration, we approximate the the first order gradient by a zeroth-order gradient estimator, and using variance reduction technique   to improve the algorithm which is similar to that in \citep{fang2018spider}. Detailedly, we choose a relatively large batch size in zeroth-order gradient estimator every $q$ iterations, while a small batch size at other iterations. The detailed algorithm is shown as in Algorithm \ref{algo2}. 
\begin{algorithm}[t]
	\caption{(ZO-VRAGDA)}
	\label{algo2}
	\begin{algorithmic}
		\item[Step 1]~ Input $x_{0}, y_{0},q$, $\alpha, \beta$, $B$, $b$; Set $t=0$.
		\item[Step 2]~ If $\mod (t,q)=0$, generate $B$ samples, i.e., $\mathcal{B}_{t}=\{\xi_{t}^{i}\}_{i=1}^{B}$
		and compute 
		\begin{align}
			m_{t}=\hat{\nabla}_{x} G(x_{t}, y_{t} ; \mathcal{B}_{t});\label{mq}
		\end{align}			
		Otherwise, generate $b$ samples, i.e., $\mathcal{I}_{t}=\{\xi_{t}^{i}\}_{i=1}^{b}$ and compute 
		\begin{align}
			m_{t}=\hat{\nabla}_{x} G(x_{t}, y_{t};{\mathcal{I}_{t}})-\hat{\nabla}_{x} G(x_{t-1}, y_{t-1};{\mathcal{I}_{t}})+m_{t-1};\label{mt}
		\end{align}
		Update $x_{t+1}$:
		\begin{equation}\label{algolstep_x2}
			x_{t+1}=x_{t}-\alpha m_{t}.
		\end{equation}
		\item[Step 3]~  If $\mod (t,q)=0$, generate $B$ samples, i.e., $\bar{\mathcal{B}}_{t}=\{\zeta_{t}^{i}\}_{i=1}^{B}$
		and compute 
		\begin{align}
			n_{t}=\hat{\nabla}_{y} G(x_{t+1}, y_{t} ; \bar{\mathcal{B}}_{t});\label{nq}
		\end{align}
		Otherwise, generate $b$ samples, i.e., $\bar{\mathcal{I}}_{t}=\{\zeta_{t}^{i}\}_{i=1}^{b}$ and compute 
		\begin{align}
			n_{t}=\hat{\nabla}_{y} G(x_{t+1}, y_{t}; \bar{\mathcal{I}}_{t})-\hat{\nabla}_{y} G(x_{t}, y_{t-1} ;\bar{\mathcal{I}}_{t})+n_{t-1}\label{nt}.
		\end{align}
		Update $y_{t+1}$:
		\begin{equation}\label{algolstep_y2}
			y_{t+1}=y_{t}+\beta n_{t}.
		\end{equation}
		\item[Step 4]~ If converges, stop; otherwise, set $t=t+1$, go to Step 2.
	\end{algorithmic}
\end{algorithm}
Note that the first-order version of Algorithm \ref{algo2} is similar to the VR-SMDA algorithm proposed in \citep{huang2021efficient}.  The main difference is that $y_{t}$ in Algorithm \ref{algo2} is updated by $x_{t}$ instead of $x_{t-1}$ that used in the VR-SMDA algorithm. In other words, Algorithm \ref{algo2} is a zeroth-order alternating GDA algorithm, whereas the VR-SMDA algorithm is a first-order simultaneous GDA algorithm. 

Next, we give some mild assumptions for \eqref{min_max_problem_stochastic}  which are also used in \citep{yang22b}.
\begin{assumption}\label{psi} 
	$ g(x,y)$ satisfies all the assumptions in Assumption \ref{a3}.
\end{assumption}
\begin{assumption}\label{a2}
	The variance of the zeroth-order stochastic gradient estimator is bounded, i.e., there exists a constant  $\sigma>0 $ such that for all $x$ and $y$, it has 
	\begin{align*}
		\mathbb{E}_{(u, \xi)}\|\hat{\nabla}_{x} G(x, y ; \xi)-\nabla_{x} g_{\mu_{1}}(x, y)\|^{2} &\leqslant\sigma^{2},\\
		\mathbb{E}_{(v, \zeta)}\|\hat{\nabla}_{y} G(x, y ; \zeta)-\nabla_{y} g_{\mu_{2}}(x, y)\|^{2} &\leqslant \sigma^{2}. 
	\end{align*}
\end{assumption}
By Assumption \ref{a2}, we can easily compute that $\mathbb{E}\|\hat{\nabla}_x G(x, y ; \mathcal{B})-\nabla_{x} g_{\mu_1}(x, y)\|^2\leqslant\frac{\sigma^{2}}{r}$ and $\mathbb{E}\|\hat{\nabla}_y G(x, y ; \bar{\mathcal{B}})-\nabla_{y} g_{\mu_2}(x, y)\|^2\leqslant\frac{\sigma^{2}}{r}$.
\begin{assumption}\label{vr_lip_assumption}
	For each component $G(x,y;\xi)$ has Lipschitz continuous gradients, i.e., there exists a constant $\bar{l}>0$ such that $\forall x_{1}, x_{2} \in \mathbb{R}^{d_{1}}$, $y_{1}, y_{2} \in \mathbb{R}^{d_{2}}$,
	\begin{align*}
		\|\nabla_x G(x_{1},y_{1};\xi) - \nabla_x G(x_{2},y_{2};\xi)\| \leqslant \bar{l}[\|x_{1}-x_{2}\|+\|y_{1}-y_{2}\|],\\
		\|\nabla_y G(x_{1},y_{1};\zeta) - \nabla_y G(x_{2},y_{2};\zeta)\| \leqslant \bar{l}[\|x_{1}-x_{2}\|+\|y_{1}-y_{2}\|].
	\end{align*}
\end{assumption}
\begin{lemma}\label{lem:5}
	$g(x,y)$ has Lipschitz continuous gradients with constant $\bar{l}$. Moreover, $\Psi(\cdot)$ is $\bar{L}$-smooth with $\bar{L}:=\bar{l}+\frac{\bar{l}^{2}}{2\mu}$.
\end{lemma}
\begin{proof}
	By Jensen's inequality and Assumption \ref{vr_lip_assumption}, 	$\forall x_{1}, x_{2} \in \mathbb{R}^{d_{1}}$, $y_{1}, y_{2} \in \mathbb{R}^{d_{2}}$, we have
	\begin{align*}
		&\|\nabla_x g(x_1,y_1)-\nabla_x g(x_2,y_2)\|\\
		=&\|\mathbb{E}[\nabla_x G(x_1,y_1,\xi)-\nabla_x G(x_2,y_2,\xi)]\|\\
		\leqslant&\mathbb{E}\|\nabla_x G(x_1,y_1,\xi)-\nabla_x G(x_2,y_2,\xi)\|\\
		\leqslant&\bar{l}[\|x_{1}-x_{2}\|+\|y_{1}-y_{2}\|].
	\end{align*}
	Similarly, we can prove that $\|\nabla_y g(x_1,y_1)-\nabla_y g(x_2,y_2)\|\leqslant \bar{l}[\|x_{1}-x_{2}\|+\|y_{1}-y_{2}\|].$ Moreover, we can similarly prove that $\Psi(\cdot)$ is $\bar{L}$-smooth with $\bar{L}:=\bar{l}+\frac{\bar{l}^{2}}{2\mu}$ by Lemma \ref{philip}. 
\end{proof}
Next, we analyze the iteration complexity of the ZO-VRAGDA algorithm for solving \eqref{min_max_problem_stochastic}.
\begin{definition} $\bar{x}$ is an $\epsilon$-stationary point of \eqref{min_max_problem_stochastic} if $\mathbb{E} \|\nabla\Psi(\bar{x})\|\leqslant\epsilon$.
\end{definition}
\begin{lemma} \label{vrlemma7}
	If Assumption \ref{vr_lip_assumption} holds, we have
	\begin{align}
		\mathbb{E}g(x_{t+1}, y_{t+1})\geqslant &\nonumber
		\mathbb{E}g(x_{t}, y_{t})+\frac{\beta}{2}\mathbb{E}\|\nabla_{y} g(x_{t+1}, y_{t})\|^{2}-\frac{\alpha}{2}\mathbb{E}\|\nabla_{x} g(x_{t}, y_{t})\|^{2}\\
		&\nonumber
		-\frac{\beta}{2}\mathbb{E}\|\nabla_{y} g(x_{t+1}, y_{t})-n_{t}\|^{2}+\frac{\alpha}{2}\mathbb{E}\| \nabla_{x} g(x_{t}, y_{t})-m_{t} \|^{2}\\
		&+\frac{\beta}{2}(1-\bar{l} \beta)\mathbb{E}\|n_{t}\|^{2}-\frac{\alpha}{2}(1+\alpha \bar{l} )\mathbb{E}\|m_{t}\|^{2}.
		\label{sec4:1}
	\end{align}
\end{lemma}

\begin{proof}
	Firstly, by Lemma \ref{lem:5} and \eqref{algolstep_y2}, we have
	\begin{align}
		&\quad\mathbb{E}g(x_{t+1}, y_{t+1})-\mathbb{E}g(x_{t+1}, y_{t}) \nonumber\\
		&\geqslant \mathbb{E}\langle\nabla_{y} g(x_{t+1}, y_{t}), y_{t+1}-y_{t}\rangle-\frac{\bar{l}}{2}\mathbb{E}\|y_{t+1}-y_{t}\|^2\nonumber\\
		&\nonumber
		=\mathbb{E}\langle\nabla_{y} g(x_{t+1}, y_{t}), \beta n_{t}\rangle-\frac{\bar{l}}{2} \beta^{2}\mathbb{E}\|n_{t}\|^{2}\\
		&\nonumber
		=\frac{\beta}{2}\mathbb{E}\|n_{t}\|^{2}+\frac{\beta}{2}\mathbb{E}\|\nabla_{y} g(x_{t+1}, y_{t})\|^{2}\\
		&\nonumber\quad-\frac{\beta}{2}\mathbb{E}\|\nabla_{y} g(x_{t+1},y_{t})-n_{t}\|^{2}-\frac{\bar{l} \beta^{2}}{2}\mathbb{E}\|n_t\|^{2}\\
		&=\frac{\beta}{2}\mathbb{E}\|\nabla_{y} g(x_{t}, y_{t})\|^{2}-\frac{\beta}{2}\mathbb{E}\|\nabla_{y} g(x_{t+1}, y_{t})-n_{t}\|^{2}+\frac{\beta}{2}(1-\bar{l} \beta)\mathbb{E}\| n_{t}\|^{2},
		\label{sec4:2}
	\end{align}
	where the second last equality is by a simple fact that $\langle \bar{a},\bar{b}\rangle=\frac{1}{2}\|\bar{a}\|^{2}+\frac{1}{2}\|\bar{b}\|^{2}-\frac{1}{2}\|\bar{a}-\bar{b}\|^{2}$, $\forall \bar{a},\bar{b}$. Similarly, by Lemma \ref{lem:5} and \eqref{algolstep_x2}, we obtain
	\begin{align}
		&\nonumber
		\mathbb{E}g(x_{t+1}, y_{t})-\mathbb{E}g(x_{t}, y_{t}) \\ \nonumber
		\geqslant&\mathbb{E}\langle\nabla_{x} g(x_{t}, y_{t}), x_{t+1}-x_{t}\rangle-\frac{\bar{l}}{2}\mathbb{E}\|x_{t+1}-x_{t}\|^{2}\\
		\nonumber
		=&-\mathbb{E}\langle\nabla_{x}g(x_{t}, y_{t}), \alpha m_{t}\rangle-\frac{\bar{l}}{2} \alpha^{2}\mathbb{E}\|m_{t}\|^{2}\\
		\nonumber
		=&\frac{\alpha}{2}\mathbb{E}\|\nabla_{x} g(x_{t}, y_{t})- m_{t}\|^{2}-\frac{\alpha}{2}\mathbb{E}\|\nabla_{x} g(x_{t}, y_{t})\|^{2}-\frac{\alpha}{2}\mathbb{E}\|m_{t}\|^{2}-\frac{\bar{l}}{2} \alpha^{2}\mathbb{E}\|m_{t}\|^{2}\\
		=&\frac{\alpha}{2}\mathbb{E}\|\nabla_{x} g(x_{t}, y_{t})-m_{t}\|^{2}-\frac{\alpha}{2}\mathbb{E}\|\nabla_{x} g(x_{t}, y_{t})\|^{2}-\frac{\alpha}{2}(1+\alpha \bar{l})\mathbb{E}\| m_{t} \|^{2}.
		\label{sec4:3}
	\end{align}
	We complete the proof by adding \eqref{sec4:2} and \eqref{sec4:3}. 
\end{proof}

\begin{lemma}\label{vrgradient_bound}
	Suppose Assumptions \ref{a2} and \ref{vr_lip_assumption} hold.  Set $p_{t}=\lceil t / q \rceil$.  We have
	\begin{align}
		&\nonumber\mathbb{E}\|\nabla_{x} g_{\mu}(x_{t}, y_{t})-m_{t}\|^{2}\\
		&\leqslant \frac{3d_{1}\bar{l}^{2}}{b} \sum_{i=(p_{t}-1)q}^{t-1}(\mathbb{E}\|x_{i+1}-x_{i}\|^{2}+\mathbb{E}\|y_{i+1}-y_{i}\|^{2})+\sum_{i=(p_{t}-1)q}^{t-1} \frac{3\bar{l}^{2} \mu_{1}^{2}d_{1}^{2}}{2b}+\frac{\sigma^{2}}{B},
		\label{vrgradient_bound_x}\\
		&\nonumber\mathbb{E}\|\nabla_{y} g_{\mu}(x_{t+1}, y_{t})-n_{t}\|^{2}\\
		&\leqslant\frac{3d_{2}\bar{l}^{2}}{b} \sum_{i=(p_{t}-1)q}^{t-1}(\mathbb{E}\|x_{i+2}-x_{i+1}\|^{2}+\mathbb{E}\|y_{i+1}-y_{i}\|^{2})+\sum_{i=(p_{t}-1)q}^{t-1} \frac{3\bar{l}^{2} \mu_{2}^{2}d_{2}^{2}}{2b}+\frac{\sigma^{2}}{B}.
		\label{vrgradient_bound_y}
	\end{align}
	
\end{lemma}

\begin{proof}
	Denote $\nabla_xG_{I}(x,y)=\nabla_{x} g_{\mu}(x, y)-\hat{\nabla}_{x} G(x, y ; I)$ for a given index set $I$. Firstly, by \eqref{mt}, for any index $j_t$ that satisfies $(p_{t}-1)q+1\leq j_t \leq t-1$,  we have
	\begin{align}
		&\mathbb{E} \|\nabla_{x} g_{\mu}(x_{j_t},y_{j_t})-m_{j_t} \|^{2}\nonumber \\
		=&\mathbb{E}\|\nabla_{x} g_{\mu}(x_{j_t}, y_{j_t})-m_{j_t-1}-(m_{j_t}-m_{j_t-1})\|^{2}\nonumber\\
		=&\mathbb{E} \|\nabla_{x} g_{\mu}(x_{j_t-1}, y_{j_t-1})-m_{j_t-1}+\nabla_xG_{I_{j_t}}(x_{j_t}, y_{j_t})-\nabla_xG_{I_{j_t}}(x_{j_t-1}, y_{j_t-1})\|^{2}\nonumber\\
		=&\mathbb{E}\|\nabla_{x} g_{\mu}(x_{j_t-1}, y_{j_t-1})-m_{j_t-1}\|^{2}\nonumber\\
		&+\mathbb{E} \| \nabla_xG_{I_{j_t}}(x_{j_t}, y_{j_t})-\nabla_xG_{I_{j_t}}(x_{j_t-1}, y_{j_t-1})\|^{2},\label{add:4.1}
	\end{align}
	where the last equality is due to $\mathbb{E}[\nabla_xG_{I_{j_t}}(x,y)]=0$. Note that if $\left\{\xi_i\right\}_{i=1}^b$ are i.i.d. random variables with zero mean, then  $\mathbb{E}\|\frac{1}{b} \sum_{i=1}^{b} \xi_{i}\|^{2}=\frac{1}{b} \mathbb{E}\|\xi_{i}\|^{2}$ for any $i \in \{1,\cdots,b\}$. Then, by using this fact, for any element $\xi_{j_{t}}^{\prime}\in I_{j_t}$, we have
	\begin{align}
		& \mathbb{E}\|\nabla_xG_{I_{j_t}}(x_{j_t}, y_{j_t})-\nabla_xG_{I_{j_t}}(x_{j_t-1}, y_{j_t-1})\|^{2}\nonumber\\
		=&\frac{1}{b} \mathbb{E} \| \nabla_{x} g_{\mu}(x_{j_t}, y_{j_t})-\hat{\nabla}_{x} G(x_{j_t},y_{j_t};\xi_{j_t}^{\prime})\nonumber\\
		&-\nabla_{x} g_{\mu}(x_{j_t-1}, y_{j_t-1})+\hat{\nabla}_{x} G(x_{j_t-1}, y_{j_t-1};\xi_{j_{t}}^{\prime})\|^{2}\nonumber\\
		\leqslant& \frac{1}{b} \mathbb{E} \| \hat{\nabla}_{x} G(x_{j_t}, y_{j_t};\xi_{j_{t}}^{\prime})-\hat{\nabla}_{x} G(x_{j_t-1},x_{j_t-1};\xi_{j_t}^{\prime})\|^{2},\label{add:4.2}
	\end{align}
	where the last inequality is due to $\mathbb{E}\|\xi-\mathbb{E}[\xi]\|^{2}=\mathbb{E}\|\xi\|^{2}-\|\mathbb{E}[\xi]\|^{2}\leqslant\mathbb{E}\|\xi\|^{2}$. On the other hand, by Assumption \ref{vr_lip_assumption} and (102) in the proof of Lemma 29 in \citep{huang2020accelerated}, we have 
	\begin{align}
		& \mathbb{E} \| \hat{\nabla}_{x} G(x_{j_t}, y_{j_t};\xi_{j_t}^{\prime})-\hat{\nabla}_{x} G(x_{j_t-1},y_{j_t-1};\xi_{j_t}^{\prime})\|^{2} \nonumber\\
		&\leqslant \frac{3\bar{l}^{2} \mu_{1}^{2} d_{1}^{2}}{2}+3 d_{1}\bar{l}^{2} \mathbb{E}(\|x_{j_t}-x_{j_t-1}\|^{2}+\|y_{j_t}-y_{j_t-1}\|^{2}).\label{add:4.3}
	\end{align}
	Combing \eqref{add:4.1}, \eqref{add:4.2} and \eqref{add:4.3}, we obtain
	\begin{align}
		&\nonumber\mathbb{E}\|\nabla_{x} g_{\mu}(x_{j_t}, y_{j_t})-m_{j_t}\|^{2}\\\nonumber
		&
		\leqslant\mathbb{E}\|\nabla_{x}g_{\mu}(x_{j_t-1}, y_{j_t-1})-m_{j_t-1}\|^{2}\\&
		\quad+\frac{1}{b}[\frac{3 \bar{l}^{2}\mu_{1}^{2}d_{1}^{2}}{2}+3d_{1}\bar{l}^{2}\mathbb{E} (\|x_{j_t}- x_{j_t-1}\|^{2}+\|y_{j_t}-y_{j_t-1} \|^{2})].
		\label{sec4:4}			
	\end{align}
	Similar to the proof of \eqref{sec4:4}, we have		
	\begin{align}
		&\mathbb{E} \| \nabla_{y} g_{\mu}(x_{j_t+1}, y_{j_t})-n_{j_t} \|^{2}\nonumber\\
		\leqslant&\nonumber\mathbb{E}\|\nabla_{y}g_{\mu}(x_{j_t}, y_{j_t-1})-n_{j_t-1}\|^{2}\\
		&
		+\frac{1}{b}[\frac{3 \bar{l}^{2}\mu_{2}^{2}d_{2}^{2}}{2}+3d_{2}\bar{l}^{2}\mathbb{E} (\|x_{j_t+1}- x_{j_t}\|^{2}+\|y_{j_t}-y_{j_t-1} \|^{2})].
		\label{sec4:5}
	\end{align}
		%
	Telescoping \eqref{sec4:4} and \eqref{sec4:5} over $j_t$ from $(p_{t}-1)q+1$ to $t$, we have
	\begin{align}
		&\quad\nonumber\mathbb{E}\|\nabla_{x} g_{\mu}(x_{t}, y_{t})-m_{t}\|^{2}\\
		&\nonumber\leqslant \mathbb{E}\|\nabla_{x}g_{\mu}(x_{(p_{t}-1)q}, y_{(p_{t}-1)q})-m_{(p_{t}-1)q}\|^{2}\\
		&+\frac{3d_{1} \bar{l}^{2}}{b}\sum_{i=(p_{t}-1)q}^{t-1}(\mathbb{E}\|x_{i+1}-x_{i}\|^{2}+\mathbb{E}\|y_{i+1}-y_{i}\|^{2}) +\sum_{i=(p_{t}-1)q}^{t-1} \frac{3l^{2} \mu_{1}^{2}d_{1}^{2}}{2b}\label{add:4.4}
	\end{align}
	and
	\begin{align}
		&\quad\nonumber\mathbb{E}\|\nabla_{y} g_{\mu}(x_{t+1}, y_{t})-n_{t}\|^{2}\\
		&\nonumber\leqslant\mathbb{E}\|\nabla_{y}g_{\mu}(x_{(p_{t}-1)q+1}, y_{(p_{t}-1)q})-n_{(p_{t}-1)q}\|^{2}\\
		&+ \frac{3d_{2} \bar{l}^{2}}{b}\sum_{i=(p_{t}-1)q}^{t-1}(\mathbb{E}\|x_{i+2}-x_{i+1}\|^{2}+\mathbb{E}\|y_{i+1}-y_{i}\|^{2})+\sum_{i=(p_{t}-1)q}^{t-1} \frac{3l^{2} \mu_{2}^{2}d_{2}^{2}}{2b}.\label{add:4.5}
	\end{align}
	After plugging \eqref{mq} and \eqref{nq} into \eqref{add:4.4} and \eqref{add:4.5} respectively, by Assumption \ref{a2}, and combing \eqref{add:4.4} and \eqref{add:4.5}, we complete the proof.
\end{proof}

\begin{lemma}
	Suppose Assumptions \ref{psi}, \ref{a2} and \ref{vr_lip_assumption} hold. Let $F_{t}=\frac{3}{2} \Psi(x_{t})-\frac{1}{2} g(x_{t}, y_{t})$. Then we have 
	\begin{align}
		&\nonumber
		\mathbb{E}F_{t}-\mathbb{E}F_{t+1}\\
		\nonumber
		\geqslant&\frac{\alpha}{4}\mathbb{E}\|\nabla \Psi(x_{t})\|^{2}+[\frac{\alpha}{2}(1-2\alpha \bar{L})-\frac{\beta}{4} \bar{l}^{2} \alpha^{2}]\mathbb{E}\|m_{t}\|^{2}\\
		&\nonumber
		-\frac{15\alpha d_{1}\bar{l}^{2}}{2b} \sum_{i=(p_{t}-1)q}^{t-1}(\atwo\mathbb{E}\|m_{i}\|^{2}+\beta^{2}\mathbb{E}\|n_{i}\|^{2})\\
		&\nonumber
		-\frac{3\beta d_{2}\bar{l}^{2}}{2b} \sum_{i=(p_{t}-1)q}^{t-1}[(2\atwo+6d_{1}\bar{l}^{2}\alpha^{4}) \mathbb{E}\|m_{i}\|^{2}+(6d_{1}\bar{l}^{2}\atwo\beta^{2}+\beta^{2}) \mathbb{E}\|n_{i}\|^{2}]\\
		&\nonumber
		+(\frac{\beta}{8} -2\alpha\kappa^{2})\mathbb{E}\|\nabla_{y} g(x_{t}, y_{t})\|^{2}	+\frac{\beta}{4} (1-\bar{l}\beta)\mathbb{E}\|n_{t}\|^{2}-\frac{5\alpha\mu_{1}^{2} \bar{l}^{2} d_{1}^{2}}{8}- \frac{\beta\mu_{2}^{2}\bar{l}^{2}d_{2}^{2}}{8}\\
		&
		-\sum_{i=(p_{t}-1)q}^{t-1} (\frac{9\mu_{1}^{2}\atwo\beta d_{1}^{2}d_{2}\bar{l}^{4}}{2b}+\frac{15\alpha \bar{l}^{2} \mu_{1}^{2}d_{1}^{2}}{4b}+\frac{3\beta \bar{l}^{2} \mu_{2}^{2}d_{2}^{2}}{4b})	-(\frac{5\alpha}{2}+\frac{\beta}{2})\frac{\sigma^{2}}{B}.
		\label{add:6}
	\end{align}
\end{lemma}

\begin{proof}
	By $\bar{L}$-smoothness of $\Psi$ in Lemma \ref{lem:5}, \eqref{algolstep_x2} and $\langle \bar{a},\bar{b}\rangle=\frac{1}{2}\|\bar{a}\|^{2}+\frac{1}{2}\|\bar{b}\|^{2}-\frac{1}{2}\|\bar{a}-\bar{b}\|^{2}$, we get
	\begin{align}
		&\nonumber\mathbb{E}\Psi(x_{t+1})\\
		\nonumber \leqslant& \mathbb{E}\Psi(x_{t})+\mathbb{E}\langle\nabla \Psi(x_{t}), x_{t+1}-x_{t}\rangle+\frac{\bar{L}}{2}\mathbb{E}\|x_{t+1}-x_{t}\|^{2}\\
		\nonumber
		=&\mathbb{E}\Psi(x_{t})-\mathbb{E}\langle\nabla\Psi(x_{t}), \alpha m_{t}\rangle+\frac{\bar{L}}{2} \alpha^{2}\mathbb{E}\|m_{t}\|^{2}\\
		\nonumber
		=&\mathbb{E}\Psi(x_{t})+\frac{\alpha}{2}\mathbb{E}\|\nabla \Psi(x_{t})-m_{t}\|^{2}-\frac{\alpha}{2}\mathbb{E}\|\nabla\Psi (x_{t})\|^{2}-\frac{\alpha}{2}\mathbb{E}\| m_{t}\|^{2}+\frac{\bar{L}}{2} \alpha^{2}\mathbb{E}\| m_{t} \|^{2}\\
		\nonumber
		=&\mathbb{E}\Psi(x_{t})+\frac{\alpha}{2}\mathbb{E}\|\nabla \Psi(x_{t})-m_{t}\|^{2}-\frac{\alpha}{2}\mathbb{E}\|\nabla \Psi(x_{t})\|^{2}-\frac{\alpha}{2}(1-\alpha \bar{L})\mathbb{E}\|m_{t}\|^{2}\\
		\nonumber
		\leqslant&\mathbb{E}\Psi(x_{t})+\alpha\mathbb{E}\|\nabla \Psi(x_{t})-\nabla_{x} g(x_{t}, y_{t})\|^{2}+\alpha\mathbb{E}\|\nabla_{x} g(x_{t}, y_{t})-m_{t}\|^{2}\\
		&-\frac{\alpha}{2}\mathbb{E}\|\nabla \Psi(x_{t})\|^{2}-\frac{\alpha}{2}(1-\alpha \bar{L})\mathbb{E}\|m_{t} \|^{2},
		\label{sec4:6}
	\end{align}
	where in the last inequality we use the Cauchy-Schwarz inequality. Next, by mutiplying $\frac{3}{2}$ and $-\frac{1}{2}$ on both sides of \eqref{sec4:6} and \eqref{sec4:1} respectively and then adding them together, after rearranging the terms and by the definition of $F_{t}$, we can obtain that
	\begin{align}
		&\nonumber
		\mathbb{E}F_{t}-\mathbb{E}F_{t+1}\\
		\nonumber
		\geqslant&\frac{3\alpha}{4}\mathbb{E}\|\nabla \Psi(x_{t})\|^{2}+\frac{\alpha}{4}(2-3\alpha \bar{L}-\alpha \bar{l})\mathbb{E}\|m_{t}\|^{2} -\frac{\alpha}{4} \mathbb{E}\|\nabla_{x} g(x_{t}, y_{t})\|^{2}\\
		\nonumber
		&-\frac{3\alpha}{2}\mathbb{E}\|\nabla \Psi(x_{t})-\nabla_{x} g(x_{t}, y_{t})\|^{2}-\frac{5\alpha}{4}\mathbb{E}\|\nabla_{x} g(x_{t}, y_{t})-m_{t}\|^{2} \\
		\nonumber
		&+\frac{\beta}{4} \mathbb{E}\|\nabla_{y} g(x_{t+1},y_{t})\|^{2}-\frac{\beta}{4} \mathbb{E}\|\nabla_{y} g(x_{t+1}, y_{t})-n_{t}\|^{2}+\frac{\beta}{4} (1-\bar{l} \beta)\mathbb{E}\|n_{t}\|^{2}\\
		\nonumber
		\geqslant&\frac{\alpha}{4}\mathbb{E}\| \nabla \Psi( x_{t})\|^{2}+\frac{\alpha}{2}(1-2\alpha \bar{L})\mathbb{E}\| m_{t}\|^{2}-2\alpha\mathbb{E}\|\nabla\Psi(x_{t})-\nabla_{x}g(x_{t}, y_{t})\|^{2}\\
		\nonumber
		&-\frac{5\alpha}{4}\mathbb{E}\|\nabla_{x} g(x_{t}, y_{t})-m_{t} \|^{2}+\frac{\beta}{4} \mathbb{E}\|\nabla_{y} g(x_{t+1},y_{t})\|^{2}\\
		&-\frac{\beta}{4} \mathbb{E}\|\nabla_{y} g(x_{t+1}, y_{t})-n_{t}\|^{2}+\frac{\beta}{4}(1-\bar{l}\beta)\mathbb{E}\| n_{t} \|^{2},\label{add:4.10}
	\end{align}
	where the second inequality is by $\bar{l}\leqslant \bar{L}$ and the fact that $-\frac{1}{2}\|\nabla_{x} g(x_{t}, y_{t})\|^{2}\geqslant -\|\nabla \Psi(x_{t})-\nabla_{x} g(x_{t}, y_{t})\|^{2}-\|\nabla \Psi(x_{t})\|^{2}$. Next, by the Cauchy-Schwarz inequality, \eqref{algolstep_x2} and  Assumption \ref{vr_lip_assumption}, we have
	\begin{align}
		&\mathbb{E}\|\nabla_{y} g(x_{t+1},y_{t})\|^{2}\nonumber\\
		=&\mathbb{E} \|\nabla_{y} g(x_{t+1},y_{t})-\nabla_{y} g(x_{t}, y_{t})+\nabla_{y} g(x_{t},y_{t})\|^{2}\nonumber\\
		\geqslant& \frac{\mathbb{E}\|\nabla_{y} g(x_{t}, y_{t})\|^{2}}{2}-\mathbb{E}\|\nabla_{y} g(x_{t+1}, y_{t})-\nabla_{y} g(x_{t}, y_{t})\|^{2}\nonumber\\
		\geqslant& \frac{\mathbb{E}\|\nabla_{y} g(x_{t}, y_{t})\|^{2}}{2}-\bar{l}^{2}\atwo\mathbb{E}\|m_{t}\|^{2}.\label{add:4.6}
	\end{align}
	By \eqref{phigap} in Lemma \ref{g smooth}, we get
	\begin{align}
		\mathbb{E}\|\nabla\Psi(x_{t})-\nabla_{x} g(x_{t}, y_{t}) \|^{2} \leqslant \kappa^2 \mathbb{E}\|\nabla_{y} g(x_{t}, y_{t})\|^{2}.\label{add:4.7}
	\end{align}
	By the Cauchy-Schwarz inequality, we have 
	\begin{align}
		\mathbb{E}\|\nabla_{x} g(x_{t}, y_{t})-m_{t}\|^{2}&\leqslant2\mathbb{E}\|\nabla_{x} g(x_{t}, y_{t})-\nabla_{x} g_{\mu}(x_{t}, y_{t})\|^{2} \nonumber\\
		&\quad +2\mathbb{E}\|\nabla_{x} g_{\mu}(x_{t}, y_{t})-m_{t}\|^{2},\label{add:4.8}\\
		\mathbb{E}\|\nabla_{y} g(x_{t+1}, y_{t})-n_{t}\|^{2}&\leqslant 2\mathbb{E}\|\nabla_{y} g(x_{t+1}, n_{t})-\nabla_{y} g_{\mu}(x_{t+1}, y_{t})\|^{2}\nonumber\\
		&\quad +2\mathbb{E}\|\nabla_{y} g_{\mu}(x_{t+1}, y_{t})-n_{t}\|^{2}.\label{add:4.9}
	\end{align}
	By plugging \eqref{add:4.6}-\eqref{add:4.9} into \eqref{add:4.10}, we have
	\begin{align}
		&\nonumber
		\mathbb{E}F_{t}-\mathbb{E}F_{t+1}\\
		\nonumber
		\geqslant&\frac{\alpha}{4}\mathbb{E}\|\nabla \Psi(x_{t})\|^{2}+[\frac{\alpha}{2}(1-2\alpha \bar{L})-\frac{\beta}{4} \bar{l}^{2} \alpha^{2}]\mathbb{E}\|m_{t}\|^{2}\\
		&\nonumber
		-\frac{5\alpha}{2}\mathbb{E}\|\nabla_{x} g(x_{t}, y_{t})-\nabla_{x} g_{\mu}(x_{t}, y_{t})\|^{2}-\frac{5\alpha}{2}\mathbb{E}\|\nabla_{x} g_{\mu}(x_{t}, y_{t})-m_{t}\|^{2}\\
		&\nonumber
		-\frac{\beta}{2}\mathbb{E}\|\nabla_{y} g(x_{t+1}, y_{t})-\nabla_{y} g_{\mu}(x_{t+1}, y_{t})\|^{2}-\frac{\beta}{2}\mathbb{E}\|\nabla_{y} g_{\mu}(x_{t+1}, y_{t})-n_{t}\|^{2}\\
		&\nonumber
		+[\frac{\beta}{8} -2\alpha\kappa^{2}]\mathbb{E}\|\nabla_{y} g(x_{t}, y_{t})\|^{2}+\frac{\beta}{4} (1-\bar{l}\beta)\mathbb{E}\|n_{t}\|^{2}\\
		\nonumber
		\geqslant&\frac{\alpha}{4}\mathbb{E}\|\nabla \Psi(x_{t})\|^{2}+[\frac{\alpha}{2}(1-2\alpha \bar{L})-\frac{\beta}{4} \bar{l}^{2} \alpha^{2}]\mathbb{E}\|m_{t}\|^{2}\\
		&\nonumber
		-\frac{15\alpha d_{1}\bar{l}^{2}}{2b} \sum_{i=(p_{t}-1)q}^{t-1}(\mathbb{E}\|x_{i+1}-x_{i}\|^{2}+\mathbb{E}\|y_{i+1}-y_{i}\|^{2})-\sum_{i=(p_{t}-1)q}^{t-1} \frac{15\alpha \bar{l}^{2} \mu_{1}^{2}d_{1}^{2}}{4b}\\
		&\nonumber
		-\frac{3\beta d_{2}\bar{l}^{2}}{2b} \sum_{i=(p_{t}-1)q}^{t-1}(\mathbb{E}\|x_{i+2}-x_{i+1}\|^{2}+\mathbb{E}\|y_{i+1}-y_{i}\|^{2})-\sum_{i=(p_{t}-1)q}^{t-1} \frac{3\beta \bar{l}^{2} \mu_{2}^{2}d_{2}^{2}}{4b}\\
		&\nonumber
		-\frac{5\alpha\mu_{1}^{2} \bar{l}^{2} d_{1}^{2}}{8}- \frac{\beta\mu_{2}^{2}\bar{l}^{2}d_{2}^{2}}{8}+(\frac{\beta}{8} -2\alpha\kappa^{2})\mathbb{E}\|\nabla_{y} g(x_{t}, y_{t})\|^{2}+\frac{\beta}{4} (1-\bar{l}\beta)\mathbb{E}\|n_{t}\|^{2}\\
		&
		-(\frac{5\alpha}{2}+\frac{\beta}{2})\frac{\sigma^{2}}{B},
		\label{sec4:7}
	\end{align}
	where the last inequaliy is by Lemma \ref{vrgradient_bound} and \eqref{gradgap} in Lemma \ref{funconvexlip}. Next, we estimate the upper bound for $\mathbb{E}\|x_{i+2}-x_{i+1}\|^{2}$, $\mathbb{E}\|x_{i+1}-x_{i}\|^{2}$ and $\mathbb{E}\|y_{i+1}-y_{i}\|^{2}$ in the right hand side of \eqref{sec4:7} when $(p_{t}-1)q\leq i \leq t-1$.  By \eqref{algolstep_x2}, \eqref{mt} and the Cauchy-Schwarz inequality, we have
	\begin{align}
		&\nonumber\mathbb{E}\|x_{i+2}-x_{i+1}\|^{2}\\
		=&\nonumber\atwo\mathbb{E}\|m_{i+1}\|^{2}\leqslant 2\atwo\mathbb{E}\|m_{i}\|^{2}+2\atwo\mathbb{E}\| m_{i+1}-m_{i} \|^{2}\\
		=&\nonumber
		2\atwo\mathbb{E}\|m_{i}\|^{2}+2\atwo\mathbb{E} \| \hat{\nabla}_{x} G(x_{i+1}, y_{i+1} ; I_{i})-\hat{\nabla}_{x}G(x_{i}, y_{i ;} I_{i})\|^{2}\\
		\leqslant&
		2\atwo\mathbb{E}\|m_{i}\|^{2}+\frac{2\atwo}{b} \sum_{j=1}^{b}\mathbb{E}\| \hat{\nabla}_{x} G(x_{i+1},y_{i+1}; \xi_{i}^{j})-\hat{\nabla}_{x}G(x_{i}, y_{i};\xi_{i}^{j}) \|^{2},\label{add:4.11}
	\end{align}
	where the last inequality is by the fact that $\mathbb{E}\|\sum_{i=1}^{b}\eta_i\|^{2}\leqslant b\sum_{i=1}^{b}\mathbb{E}\|\eta_{i}\|^{2}$ for i.i.d random variables $\{\eta_1,\cdots, \eta_b\}$. 
	By (102) in the proof of Lemma 29 in \citep{huang2020accelerated}, we have that 
	\begin{align}
		&\mathbb{E} \| \hat{\nabla}_{x} G(x_{i+1}, y_{i+1};\xi_{i}^{\prime})-\hat{\nabla}_{x} G(x_{i},y_{i};\xi_{i}^{\prime})\|^{2}\nonumber\\ 
		&\leqslant \frac{3\bar{l}^{2} \mu_{1}^{2} d_{1}^{2}}{2}+3 d_{1}\bar{l}^{2} \mathbb{E}(\|x_{i+1}-x_{i}\|^{2}+\|y_{i+1}-y_{i}\|^{2}).\label{add:4.12}
	\end{align}
	By pluggging \eqref{add:4.12} into \eqref{add:4.11}, and using \eqref{algolstep_x2} and \eqref{algolstep_y2}, we obtain
	\begin{align}
		&	\mathbb{E}\|x_{i+2}-x_{i+1}\|^{2}	\nonumber\\
		\leqslant&(2\atwo+6d_{1}\bar{l}^{2}\alpha^{4}) \mathbb{E}\|m_{i}\|^{2}+6d_{1}\bar{l}^{2}\atwo\beta^{2} \mathbb{E}\|n_{i}\|^{2}+3\bar{l}^{2}\atwo\mu_{1}^{2}d_{1}^{2}.\label{m_i}
	\end{align}
	The proof is then completed by plugging \eqref{m_i} into \eqref{sec4:7}, and using $\mathbb{E}\|x_{i+1}-x_{i}\|^{2}=\atwo\mathbb{E}\|m_{i}\|^{2}$ and $\mathbb{E}\|y_{i+1}-y_{i}\|^{2}=\beta^{2}\mathbb{E}\|n_{i}\|^{2}$.
\end{proof}
	Set  $\mu_{1}=\frac{\epsilon}{\sqrt{35d_{1}^{2}\bar{L}^{2}+576\kappa^{2}d_{1}^{2}d_{2}\bar{L}^{2}}}=\mathcal{O}(\frac{\epsilon}{\kappa d_{1}d_{2}^{\frac{1}{2}}\bar{L}})$,  $\mu_{2}=\frac{\epsilon}{\sqrt{112\kappa^{2}d_{2}^{2}\bar{L}^{2}}}=\mathcal{O}(\frac{\epsilon}{\kappa d_{2}\bar{L}})$, $q=b=\frac{\kappa}{\epsilon}$ and $B=\frac{(40+128\kappa^{2})\sigma^{2}}{\epsilon^{2}}$.
	\begin{theorem}\label{zovrgda_theorem}
		Suppose Assumptions \ref{psi}, \ref{a2} and \ref{vr_lip_assumption} hold. If $\beta\leqslant\frac{1}{C}$ with $C=\bar{L}+30d_{1}\bar{L}+6d_{2}\bar{L}+36d_{1}d_{2}\bar{L}$, $\alpha=\frac{\beta}{16\kappa^{2}}$, we have 
		\begin{align}
			T(\epsilon)\leqslant\frac{8[3\Psi(x_{0})-g(x_0,y_0)-2\Psi^{*} ]}{\alpha\epsilon^{2}}.
		\end{align}
	\end{theorem}
	\begin{proof}
		Telescoping and rearranging \eqref{add:6}, by $\bar{l}\leqslant \bar{L}$ and $q=b$, we get
		\begin{align}
			&\nonumber
			F_{0}-\mathbb{E}F_{T(\epsilon)}\\
			\nonumber \geqslant& \frac{\alpha}{4} \sum_{t=0}^{T(\epsilon)-1}\mathbb{E}\|\nabla \Psi(x_{t}) \|^{2}+H_{1}\sum_{t=0}^{T(\epsilon)-1} \mathbb{E}\|m_{t}\|^{2}+H_{2}\sum_{t=0}^{T(\epsilon)-1} \mathbb{E}\|n_{t}\|^{2}\\
			\nonumber&+[\frac{\beta}{8} -2\alpha \kappa^{2}]\sum_{t=0}^{T(\epsilon)-1} \mathbb{E}\|\nabla_{y}g(x_{t}, y_{t}) \|^{2}-[\frac{5\alpha}{2}+\frac{\beta}{2}]\frac{\sigma^{2}}{B}T(\epsilon)- \frac{15\alpha \bar{L}^{2}  d_{1}^{2}}{4} \mu_{1}^{2}T(\epsilon)  \\
			&-\frac{3\beta \bar{L}^{2}  d_{2}^{2}}{4}\mu_{2}^{2}T(\epsilon)- \frac{5\alpha \bar{L}^{2} d_{1}^{2}}{8}\mu_{1}^{2}T(\epsilon)- \frac{\beta \bar{L}^{2} d_{2}^{2}}{8}\mu_{2}^{2}T(\epsilon)- \frac{9\atwo\beta d_{1}^{2}d_{2}\bar{L}^{4}}{2} \mu_{1}^{2}T(\epsilon), \label{add:4.13}
		\end{align}
		where $H_{1}=[\frac{\alpha}{2}(1-2\alpha \bar{L})-\frac{\alpha^{2}\beta \bar{L}^{2}}{4}   -\frac{15\alpha^{3}d_{1} \bar{L}^{2}}{2}-\frac{3\beta d_{2} \bar{L}^{2}}{2}(2\atwo+6\alpha^{4}d_{1}\bar{L}^{2})],H_{2}=[\frac{\beta}{4} (1-\bar{L} \beta)-\frac{15\alpha\beta^{2}d_{1}\bar{L}^{2}}{2}- \frac{3\beta^{3}d_{2}\bar{L}^{2}}{2}-9\atwo\beta^{3}d_{1}d_{2}\bar{L}^{4}]$.
		If $\beta\leqslant\frac{1}{C}$ and $\alpha=\frac{\beta}{16\kappa^{2}}\leqslant\frac{1}{C}$, then $H_{1}\geqslant\frac{\alpha}{2}-\alpha^{2}\bar{L}(\frac{5}{4}+\frac{15}{2}d_{1}+3d_{2}+9d_{1}d_{2})\geqslant0$, $H_{2}\geqslant\frac{\beta}{4}-\beta^{2}\bar{L}(\frac{1}{4}+\frac{15}{2}d_{1}+\frac{3}{2}d_{2}+9d_{1}d_{2})\geqslant0$ and $\frac{\beta}{8} -2\alpha \kappa^{2}=0$.
		Then from \eqref{add:4.13}, we immediately have
		\begin{align}\nonumber
			&F_{0}-\mathbb{E}F_{T(\epsilon)}\\
			\nonumber \geqslant &\frac{\alpha}{4} \sum_{t=0}^{T(\epsilon)-1}\|\nabla \Psi(x_{t}) \|^{2}-(\frac{5}{2}\alpha+8\alpha\kappa^{2})\frac{\sigma^{2}}{B}T(\epsilon)-\frac{35}{8}\alpha \bar{L}^{2} \mu_{1}^{2} d_{1}^{2}T(\epsilon)\\
			&\nonumber -14\alpha\kappa^{2}\bar{L}^{2}\mu_{2}^{2}d_{2}^{2}T(\epsilon)-72\alpha\kappa^{2}\bar{L}^{2}\mu_{1}^{2}d_{1}^{2}d_{2}T(\epsilon).
		\end{align}
		By the definition of $T(\epsilon)$, and the setting of $\mu_{1}$, $\mu_{2}$ and $B$, we have
		\begin{align}
			\nonumber\epsilon^{2}&\nonumber \leqslant\frac{4}{\alpha T(\epsilon)}\left[F_{0}-\min _{x, y} \left(\frac{3}{2} \Psi(x)-\frac{1}{2} g(x, y)\right)\right]+\frac{3}{4}\epsilon^{2}\\
			&\nonumber \leqslant\frac{4}{\alpha T(\epsilon)}[F_{0}-\Psi^{*} ]+\frac{3}{4}\epsilon^{2}\\
			& \leqslant\frac{2}{\alpha T(\epsilon)}[3\Psi(x_{0})-g(x_0,y_0)-2\Psi^{*} ]+\frac{3}{4}\epsilon^{2},\label{add:7}
		\end{align}
		where the second inequality is due to $\Psi(x)\geqslant g(x,y)$ by the definition of $\Psi(x)$. The proof is completed by \eqref{add:7}.
	\end{proof}
	
	By choosing $\beta=\frac{1}{C}$ in Theorem \ref{zovrgda_theorem}, we can easily compute that $T(\epsilon)=\mathcal{O}(\kappa^{2}d_{1}d_{2}\bar{L}\epsilon^{-2})$ which implies  the total number of function value queries are $(\frac{4B}{q}+\frac{(q-1)8b}{q})*T(\epsilon)=\mathcal{O}(\kappa^{3}d_1d_2\bar{L}\epsilon^{-3})$ by the choices of $b=\frac{\kappa}{\epsilon}$, $B=\frac{(40+128\kappa^{2})\sigma^{2}}{\epsilon^{2}}$ and $q=\frac{\kappa}{\epsilon}$ for Algorithm \ref{algo2} to find an $\epsilon$-stationary point for \eqref{min_max_problem_stochastic}.
	\section{Numerical Experiments}
	In this section, we consider a stochastic version of Algorithm~\ref{algo1}  where $\hat{\nabla}_{x} f(x_{t}, y_{t})$ and $\hat{\nabla}_{y} f(x_{t+1}, y_{t})$ are replaced by $\hat{\nabla}_{x} G(x_{t}, y_{t} ; \mathcal{B}_{t})$ and $\hat{\nabla}_{y} G(x_{t+1}, y_{t}; \bar{\mathcal{B}}_{t})$ with $\mathcal{B}_{t}=\{\xi_{t}^{i}\}_{i=1}^{B}$ and $\bar{\mathcal{B}}_{t}=\{\zeta_{t}^{i}\}_{i=1}^{B}$ respectively. We denote it as ZO-SAGDA algorithm.
	We test two numerical experiments to show the efficiency of ZO-SAGDA algorithm and ZO-VRAGDA algorithm for solving a Wasserstein GAN problem and a robust polynomial optimization problem.
	
	\subsection{Wasserstein GAN problem}
	In this section, we first consider the following WGAN problem~\citep{arjovsky2017wasserstein},
	\begin{align*}
		\min_{\varphi_1,\varphi_2}\max_{\phi_1,\phi_2} & \ f(\varphi_1,\varphi_2, \phi_1, \phi_2)  \triangleq \mathbb{E}_{(x^{real},z)\sim \mathcal{D}}(D_{\phi}(x^{real})-D_{\phi}(G_{\varphi_1,\varphi_2}(z))) - \lambda \|\phi\|^2,
	\end{align*}
	where  $G_{\varphi_1,\varphi_2}(z) = \varphi_1 + \varphi_2 z$, $D_{\phi}(x) =\phi_1 x + \phi_2 x^2$, $\phi=(\phi_1,\phi_2)$, $x^{real}$ is generated from a normal distribution with mean $\varphi_1^{*}=0$ and variance $\varphi_2^{*}=0.1$ which are also the optimal solutions, and variable $z$ is generated from a normal distribution with mean $0$ and variance $1$. Set $\lambda = 0.001$ which is the same as that in \citep{yang22b}.
	\begin{figure}[t]
		\centering
		\includegraphics[width=350pt]{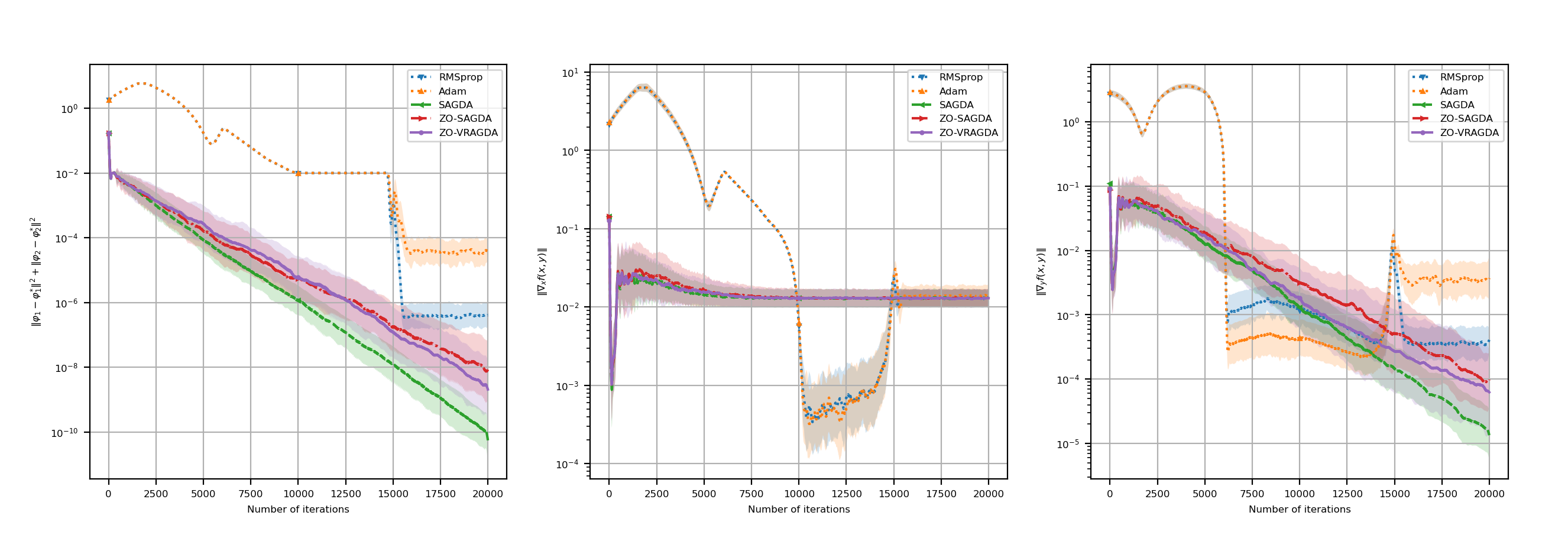}
		\caption{Performance of five tested algorithms for solving WGAN problem.}\label{img_compare}
	\end{figure}
	
	We compare ZO-SAGDA algorithm and ZO-VRAGDA algorithm with three first-order algorithms, i.e., SAGDA~\citep{yang22b}, Adam~\citep{kingma2014adam} and RMSprop~\citep{tieleman2012lecture}. Set $B=100$ for all five tested algorithms. We set $q=b=10$, $\alpha=0.1$, $\beta=0.5$ in ZO-VRAGDA, $\alpha=0.1$, $\beta=0.5$ in ZO-SAGDA and $\tau_{1}=0.1$, $\tau_{2}=0.5$ in SAGDA respectively. All the parameters of Adam algorithm and RMSprop algorithm are chosen the same as that in \citep{kingma2014adam} and \citep{tieleman2012lecture} respectively. 
	
	Figure \ref{img_compare} shows the average distance between $\varphi$ and $\varphi^{*}$, the average change of the gradient of the objective function with respect to $x$ and $y$ respectively of all the five test algorithms as the number of iterations changes over 10 independent runs. The shaded part around 5 curves denotes the standard deviation. From Figure \ref{img_compare}, we can find that both the proposed ZO-SAGDA algorithm and the ZO-VRAGDA algorithm outperform Adam algorithm and RMSprop algorithm, and approximate the performance of SAGDA algorithm which is a first-order GDA algorithm. 
	
	\subsection{Robust Polynomial Optimization Problem}
	We consider the following robust polynomial optimization problem \citep{dimitris2010robust},
	\begin{align}
		&\underset{\mathbf{x} \in \mathcal{C}}{\operatorname{max}}\  \underset{\|\boldsymbol{y}\|_2 \leq 0.5}{\operatorname{min}} f(\mathbf{x},\boldsymbol{y}):=-2\left(x_1-y_1\right)^6+12.2\left(x_1-y_1\right)^5-21.2\left(x_1-y_1\right)^4 \nonumber\\
		&-6.2\left(x_1-y_1\right)+6.4\left(x_1-y_1\right)^3+4.7\left(x_1-y_1\right)^2-\left(x_2-y_2\right)^6 \nonumber\\
		&+11\left(x_2-y_2\right)^5-43.3\left(x_2-y_2\right)^4+10\left(x_2-y_2\right)+74.8\left(x_2-y_2\right)^3 \nonumber\\
		&-56.9\left(x_2-y_2\right)^2+4.1\left(x_1-y_1\right)\left(x_2-y_2\right)+0.1\left(x_1-y_1\right)^2\left(x_2-y_2\right)^2 \nonumber\\
		&-0.4\left(x_2-y_2\right)^2\left(x_1-y_1\right)-0.4\left(x_1-y_1\right)^2\left(x_2-y_2\right),
	\end{align}
	where $\mathcal{C}=\{x_1 \in (-0.95,3.2), x_2 \in (-0.45,4.4)\}$.
	
	We use the following regret function versus iteration $t$ to measure the quality of the solution obtained by three tested algorithms, which is also used in~\citep{liu2020min}, i.e.,
	\begin{align}
		\operatorname{Regret}(t)=\underset{\|\boldsymbol{y}\|_2 \leq 0.5}{\operatorname{min}} f\left(\mathbf{x}^*,\boldsymbol{y}\right)-\underset{\|\boldsymbol{y}\|_2 \leq 0.5}{\operatorname{min}} f(\mathbf{x}^{(t)},\boldsymbol{y}),
	\end{align}
	where $\mathbf{x}^{(t)}$ is the $t$th iteration point generated by the tested algorithm, $\mathbf{x}^*=[-0.195,0.284]^T$ and $\min_{\|\boldsymbol{y}\|_2 \leq 0.5} f(\mathbf{x}^*,\boldsymbol{y})=-4.33$.
	
	\begin{figure}[t]
		\centering
		\subfigure{\includegraphics[width=150pt]{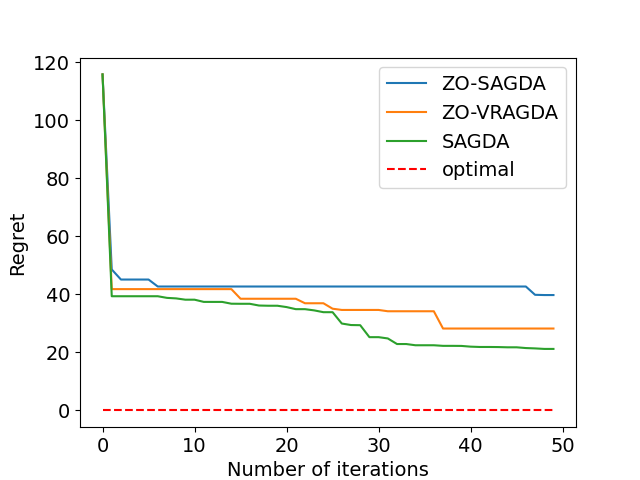}}
		\subfigure{\includegraphics[width=150pt]{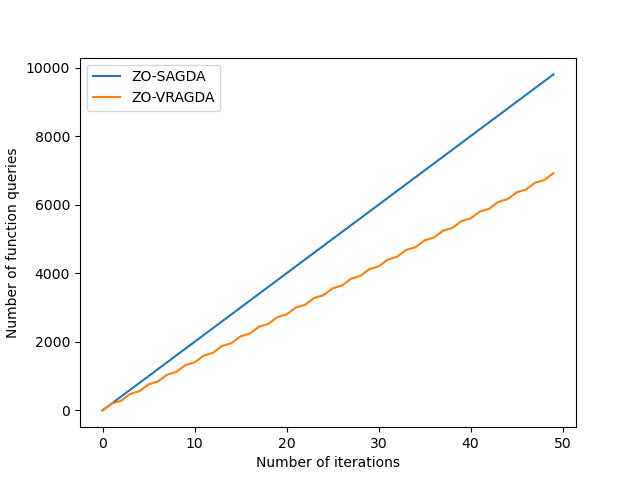}}
		\caption{Performance of three algorithms for solving robust polynomial optimization problem.}
		\label{figs}
	\end{figure}
	
	We compare ZO-SAGDA algorithm and ZO-VRAGDA algorithm  with the SAGDA algorithm~\citep{yang22b}, which is a first-order algorithm.  Set $B=50$ for all three tested algorithms. We set $q=2$, $b=10$, $\alpha=0.1$, $\beta=0.1$ in ZO-VRAGDA, $\alpha=0.1$, $\beta=0.1$ in ZO-SAGDA and $\tau_{1}=0.1$, $\tau_{2}=0.1$ in SAGDA respectively. Note that, for ZO-SAGDA algorithm and ZO-VRAGDA algorithm, instead of the exact function value, we use the noisy function value with an additional normal distribution random noise with mean $0$ and variance $0.5$ at each iteration. 
	
	Figure \ref{figs} shows the average minimum achieved regret function value and the number of function queries up to the $t$th iteration over 5 independent runs, which is the same as that used in~\citep{liu2020min}.  From Figure \ref{figs}, we can find that the performance of ZO-SAGDA algorithm and ZO-VRAGDA algorithm is similar to that of SAGDA algorithm, which is a first order algorithm. Moreover, the number of function value computation of ZO-VRAGDA algorithm is much less than that of ZO-SAGDA algorithm at each iteration. 
	
	\section{Conclusions and Discussions}
	In this paper, we propose a zeroth-order alternating gradient descent ascent (ZO-AGDA) algorithm and a zeroth-order variance reduced alternating gradient descent ascent (ZO-VRAGDA) algorithm  for solving a class of nonconvex-nonconcave minimax problems, i.e., NC-PL minimax problem, under the deterministic and the stochastic setting respectively. The total number of function value queries to obtain an $\epsilon$-stationary point of ZO-AGDA and ZO-VRAGDA algorithm for solving NC-PL minimax problem is upper bounded by $\mathcal{O}(\varepsilon^{-2})$ and $\mathcal{O}(\varepsilon^{-3})$ respectively, both of which match the iteration complexity of the corresponding first-order algorithms. To the best of our knowledge, ZO-AGDA and ZO-VRAGDA are the first two zeroth-order algorithms with the complexity gurantee for solving NC-PL minimax problems. Our numerical results further demonstrate the efficiency of the proposed algorithms which approximate the performance of the corresponding first-order algorithms.
	
	Furthermore, note that there is another zeroth-order gradient estimator based on Gaussian smoothing technique that proposed in \citep{yurii17random} which can be used to estimate the true gradient similar as that in \eqref{sec2:1} and \eqref{sec2:2}. For the corresponding algorithms, we also can obtain the same total complexity result with Theorem \ref{zoagda_theorem} and Theorem \ref{zovrgda_theorem} that shown in Section 3 and Section 4 respectively. For the brevity of the article, we omit the detailed proofs. 
	
	For more general nonconvex-nonconcave minimax problems that the PL conditions are not satisfied, it is worthy of further in-depth study whether the iteration complexity of the proposed algorithm can be guaranteed or not.

\vskip 0.2in
\bibliography{zoncpl}

\end{document}